\documentclass[11pt]{amsart}
\usepackage{amssymb}
\usepackage[margin=1.07in]{geometry}
\newtheorem{thm}{Theorem}[section]

\newtheorem{prop}[thm]{Proposition}
\newtheorem{cor}[thm]{Corollary}
\theoremstyle{definition}
\newtheorem{defn}[thm]{Definition}

\theoremstyle{remark}
\newtheorem{remark}[thm]{Remark}
\numberwithin{equation}{section}

\begin{document}

\title[Projective kernels and replicability]{Projective Kernels and Replicability in Modular Function Theory}
\author{Hicham Saber}
\author{Abdellah Sebbar}
\address{Department of Mathematics, University of Ha'il, Saudi Arabia}
\address{Department of Mathematics and Statistics, University of Ottawa, Ottawa, Ontario K1N 6N5, Canada}
\email{hi.saber@uoh.edu.sa}
\email{asebbar@uottawa.ca}


\begin{abstract}
We introduce a projective-kernel framework for the study of replicable functions in modular function theory. The main point is that the same two-point kernel simultaneously encodes the differential projective geometry of a modular function and the Faber--Grunsky data governing its replicability. This makes it possible to translate between Schwarzian invariants, coefficient identities, and modular correspondences within a single structure.

The kernel yields a reconstruction theorem showing that the ordinary Schwarzian determines the normalized Grunsky matrix and hence the underlying Laurent expansion. When Norton's replicability relations are imposed, the projective kernel produces strong arithmetic restrictions on the possible cusp data. In the degree-one case this leads to a precise classification: replicability and complete replicability are equivalent to solvability of the projective monodromy, while the icosahedral case is excluded by an explicit Grunsky obstruction. The same principle extends to the arithmetic Hecke triangle groups. The projective-kernel viewpoint also isolates the remaining global difficulty in Norton's Hauptmodul conjecture and provides a natural setting in which replicability, projective monodromy, and modular differential invariants can be studied together.
\end{abstract}

\maketitle

\section{Introduction}

Replicability is one of the characteristic algebraic features of genus-zero modular functions. It is usually expressed through Faber polynomials, Fourier coefficients, and Hecke-type identities, and it plays a central role in the theory of moonshine and completely replicable functions. These formulas are strikingly rigid, but their relation with the projective geometry of modular functions is not immediate. The purpose of this paper is to make that relation explicit. Our main object is a projectively invariant two-point kernel whose local and cusp expansions encode, respectively, the differential invariants of a modular function and the coefficient data that control replicability.

For a locally univalent meromorphic function $f$, consider
\begin{equation}\label{eq:intro-kernel}
\mathcal P_f(w,z)
=
\frac{f'(w)f'(z)}{(f(w)-f(z))^2}-\frac1{(w-z)^2}.
\end{equation}
The kernel is unchanged by postcomposition with an element of $\mathrm{PGL}_2(\mathbb C)$. Its Taylor expansion along the diagonal gives the Aharonov higher Schwarzians \cite{aharonov}, while its expansion in normalized cusp coordinates is governed by the Faber--Grunsky coefficients. Related links between higher Schwarzians, invariant differential operators, and Grunsky coefficients occur in \cite{harmelin,harmelin-derivatives,harmelin-grunsky,schippers,kim-sugawa,tamanoi}. The point needed here is more specific: the two expansions of \eqref{eq:intro-kernel} can be compared term by term, and this comparison converts projective differential information into the coefficient identities that enter replicability.

Suppose that $F(\tau)=f(q)$ with $q=e^{2\pi i\tau/\ell}$ at a cusp of width $\ell$, where
\[
f(q)=q^{-1}+\sum_{n\geq1}a_nq^n.
\]
If $h_{m,n}$ denotes the Faber--Grunsky coefficients, the kernel produces the moments
\[
\mathfrak G_k[F](\tau)
=
\frac{B_k}{k}
+
\sum_{m,n\geq1}m n^{k-1}h_{m,n}q^{m+n}.
\]
These moments are, up to an explicit constant, the higher Schwarzians $S_k[F]$. The Aharonov recurrence therefore becomes a differential recurrence for the Grunsky moments. Conversely, a finite Vandermonde inversion on each antidiagonal $m+n=N$ recovers the coefficients $h_{m,N-m}$ from finitely many Fourier coefficients of the moments. Since every higher moment is a differential polynomial in the second one, the ordinary Schwarzian determines the full normalized Grunsky matrix and hence the normalized Laurent expansion. Thus the projective kernel gives an explicit reconstruction principle: the differential projective datum carried by the Schwarzian determines the coefficient data on which replicability is built.

This becomes particularly restrictive when Norton's relations are imposed. Replicability identifies Grunsky coefficients according to their gcd and lcm, while projective equivariance forces a support condition controlled by the parabolic monodromy. If the image of the translation has order $h$, then
\[
h_{m,n}=0\quad\text{unless}\quad h\mid m+n.
\]
We prove that this support is compatible with Norton's equivalence precisely when every unit modulo $h$ is an involution, or equivalently when
\[
h\mid24.
\]
The cusp-width restriction therefore appears directly from the interaction between replicability and the projective kernel, without requiring an a priori genus-zero classification.

The same interaction becomes sharp in the degree-one case. The finite-monodromy Schwarzian equations relevant here were classified in the authors' earlier work \cite{forum}. Their degree-one solutions have projective images
\[
S_3,\quad A_4,\quad S_4,\quad A_5,
\]
corresponding to cusp widths $2,3,4,5$. We show that the first three cases are completely replicable, whereas the icosahedral case fails Norton's relation by an explicit Grunsky obstruction. Consequently, within this class,
\[
\text{replicability}
\quad\Longleftrightarrow\quad
\text{complete replicability}
\quad\Longleftrightarrow\quad
\text{solvability of the projective monodromy}.
\]
After normalization to width one, the replicable cases are the Hauptmoduls of $\Gamma_0(4)$, $\Gamma_0(9)$, and $\Gamma_0(16)$. In this setting the distinction between solvable and icosahedral monodromy is therefore detected by the same kernel that encodes the replication identities.

The degree-one analysis extends to Hecke triangle groups. For $H_m=\Delta(2,m,\infty)$, the genus formula forces the projective image to be the finite spherical triangle group $\Delta(2,m,h)$. For the arithmetic groups $H_3,H_4,H_6$, the same solvability criterion for replicability holds. The resulting width-one Hauptmoduls are those of
\[
\Gamma_0(4),\ \Gamma_0(8),\ \Gamma_0(9),\ \Gamma_0(12),\ \Gamma_0(16),\ \Gamma_0(18),
\]
and are completely replicable \cite{alexander-replicable,ford-mckay-norton}. For nonarithmetic Hecke groups the projective classification remains valid, but Norton's replication law is tied to the modular Hecke correspondences, so no corresponding replicability assertion is made.

The kernel also gives a useful formulation of the remaining difficulty in Norton's Hauptmodul conjecture \cite{norton-more,cummins-norton}. We prove that the Grunsky matrix determines the entire replicate family and that Norton's exceptional family $q^{-1}+aq$ is exactly the locus of diagonal Grunsky support. Since the ordinary Schwarzian reconstructs the Grunsky matrix, it already determines all replicate data. The unresolved step is global: the correspondences supplied by ordinary replicability must close to genuine modular self-correspondences. This separates the coefficient-theoretic rigidity, which is visible in the kernel, from the global descent needed to obtain a Hauptmodul. The formulation is consistent with the genus-zero theorem of Cummins and Gannon \cite{cummins-gannon} and with Carnahan's finite-order result \cite{carnahan}.

A final consequence of the same projective viewpoint concerns arbitrary Hauptmoduls. After correcting the higher Schwarzians by the logarithmic derivative $f''/(2f')$ and dividing by suitable powers of $f'$, one obtains rational functions of the Hauptmodul. The resulting hierarchy is determined recursively by the ordinary rational Schwarzian equation and extends naturally to modular differential equations. These constructions are secondary to the replicability problem, but they show that the projective kernel organizes a broader family of differential invariants around the same basic datum.

The paper is organized as follows. Sections 2 and 3 develop the higher Schwarzians and their transformation law. Section 4 identifies them with Grunsky moments and proves reconstruction from the ordinary Schwarzian. Section 5 treats replicability, the cusp-width obstruction, the degree-one classification for the modular and Hecke triangle groups, and the projective-kernel formulation of Norton's Hauptmodul conjecture. Section 6 develops the modular covariants and the rational higher-Schwarzian hierarchy for Hauptmoduls. Section 7 applies the formalism to modular differential equations.

\section{Aharonov higher Schwarzians}

Let $D$ be a domain in the Riemann sphere and let $f$ be meromorphic on $D$. At a point where $f$ is holomorphic and $f'(z)\neq0$, define
\[
G_f(w,z)=\frac{f'(z)}{f(w)-f(z)}.
\]
Near $w=z$ one has a Laurent expansion
\begin{equation}\label{eq:ahar-expansion}
G_f(w,z)
=
\frac1{w-z}-\sum_{n=1}^{\infty}S_n[f](z)(w-z)^{n-1}.
\end{equation}
The coefficients $S_n[f]$ are the Aharonov invariants. The first two are
\begin{equation*}
S_1[f]=\frac{f''}{2f'},
\quad
6S_2[f]=\{f,z\}.
\end{equation*}

Differentiating \eqref{eq:ahar-expansion} with respect to $z$ gives the recurrence
\begin{equation}\label{eq:ahar-recurrence}
(n+1)S_n[f]
=
S_{n-1}'[f]
+
\sum_{k=2}^{n-2}S_k[f]S_{n-k}[f],
\quad n\geq3.
\end{equation}
The sum is empty for $n=3$.

Differentiating \eqref{eq:ahar-expansion} with respect to $w$ gives the regular two-point kernel
\begin{equation}\label{eq:local-projective-kernel}
\mathcal P_f(w,z)
:=
\frac{f'(w)f'(z)}{(f(w)-f(z))^2}-\frac1{(w-z)^2}
=
\sum_{n=2}^{\infty}(n-1)S_n[f](z)(w-z)^{n-2}.
\end{equation}
The kernel $\mathcal P_f$ is unchanged under postcomposition of $f$ by a linear fractional transformation. Hence
\begin{equation}\label{eq:projective-invariance}
S_n[T\circ f]=S_n[f],\quad n\geq2,
\end{equation}
for every $T\in\mathrm{PGL}_2(\mathbb C)$.

The following classical relation with second-order equations will be used later. If $y_1,y_2$ are linearly independent solutions of
\begin{equation*}
y''+R(z)y=0,
\end{equation*}
then
\begin{equation*}
f=\frac{y_1}{y_2}
\quad\Longrightarrow\quad
S_2[f]=\frac{R}{3}.
\end{equation*}

\section{Quasimodular forms and equivariant functions}

Let $\Gamma$ be a discrete subgroup of $\mathrm{SL}_2(\mathbb R)$. We use meromorphic modular and quasimodular forms throughout. A meromorphic function $F$ is quasimodular of weight $k$ and depth at most $p$ if there are meromorphic functions $F_0,\ldots,F_p$ such that
\begin{equation}\label{eq:qm}
(cz+d)^{-k}F(\gamma z)
=
\sum_{r=0}^{p}F_r(z)
\left(\frac{c}{cz+d}\right)^r
\end{equation}
for every
\[
\gamma=\begin{pmatrix}a&b\\c&d\end{pmatrix}\in\Gamma.
\]
The polynomial $\sum_{r=0}^pF_rX^r$ is the associated quasimodular polynomial. When this polynomial is uniquely determined by $F$, the depth is its degree, and we write
\[
\delta F=F_1.
\]
This is the situation for the modular group and its finite-index subgroups, which are the cases used below. The operator $\delta$ lowers the weight by two and is part of the standard $\mathfrak{sl}_2$ structure on quasimodular forms; see \cite{ka-za,royer,123}. For a general discrete group we will state the transformation polynomial explicitly and only assert an upper bound for the depth.

A meromorphic function $f$ on $\mathbb H$ is $\rho$-equivariant for $\Gamma$ if there is a projective representation
\[
\rho:\Gamma\longrightarrow\mathrm{PGL}_2(\mathbb C)
\]
such that
\[
f(\gamma z)=\rho(\gamma)f(z).
\]
The classical Schwarzian criterion says that $f$ is equivariant for some $\rho$ if and only if $\{f,z\}$ is a meromorphic modular form of weight four; see \cite{mathann,forum}.

\begin{thm}\label{thm:transformation-law}
Let \(f\) be a nonconstant meromorphic function on \(\mathbb H\). The following are equivalent.
\begin{enumerate}
\item The function \(f\) is \(\rho\)-equivariant for some projective representation of \(\Gamma\).
\item The function \(S_2[f]\) is a meromorphic modular form of weight \(4\) for \(\Gamma\).
\item For every \(n\geq2\), the higher Schwarzian \(S_n[f]\) is a meromorphic quasimodular form of weight \(2n\) and depth at most \(n-2\) for \(\Gamma\).
\end{enumerate}
When these conditions hold, for every \(n\geq2\) and every
\[
\gamma=
\begin{pmatrix}
a&b\\
c&d
\end{pmatrix}
\in\Gamma,
\]
one has
\begin{equation}\label{eq:universal-transformation}
(cz+d)^{-2n}S_n[f](\gamma z)
=
\sum_{j=0}^{n-2}\binom{n-2}{j}S_{n-j}[f](z)
\left(\frac{c}{cz+d}\right)^j.
\end{equation}
The corresponding transformation polynomial is
\[
P_n(X)
=
\sum_{j=0}^{n-2}\binom{n-2}{j}S_{n-j}[f]X^j.
\]
Its degree is at most \(n-2\), and its coefficient of \(X^{n-2}\) is \(S_2[f]\). Whenever the associated quasimodular polynomial is unique, \(S_n[f]\) therefore has exact depth \(n-2\) if \(S_2[f]\not\equiv0\). If \(S_2[f]\equiv0\), then \(f\) is a linear fractional function and \(S_n[f]\equiv0\) for every \(n\geq2\).
\end{thm}

\begin{proof}
The equivalence of \({\rm (1)}\) and \({\rm (2)}\) is the classical Schwarzian criterion. For
\(
\gamma(z)=(az+b)/(cz+d)
\)
one obtains from the defining expansion of \(G_f\), using
\[
\gamma(w)-\gamma(z)=\frac{w-z}{(cw+d)(cz+d)},
\quad
\gamma'(z)=\frac1{(cz+d)^2},
\]
the composition formula
\begin{equation}\label{eq:composition-mobius}
S_n[f\circ\gamma](z)
=
\sum_{j=0}^{n-2}\binom{n-2}{j}
\frac{(-c)^j}{(cz+d)^{2n-j}}
S_{n-j}[f](\gamma z).
\end{equation}
Indeed, this follows by expanding
\(
(1+c(w-z)/(cz+d))^{1-k}
\)
and comparing the coefficient of \((w-z)^{n-1}\).

If \(f\) is \(\rho\)-equivariant, then
\(
f\circ\gamma^{-1}=\rho(\gamma^{-1})\circ f
\), so projective invariance gives
\(
S_n[f\circ\gamma^{-1}]=S_n[f]
\). Applying \eqref{eq:composition-mobius} to \(\gamma^{-1}\) and evaluating at \(\gamma z\) yields \eqref{eq:universal-transformation}. Thus \({\rm (1)}\) implies \({\rm (3)}\). Conversely, the case \(n=2\) in \({\rm (3)}\) says that \(S_2[f]\) has weight four and depth zero, hence is modular, proving \({\rm (3)}\Rightarrow{\rm (2)}\).

Finally, the coefficient of \(X^{n-2}\) in the transformation polynomial is \(S_2[f]\). This gives the exact-depth assertion under uniqueness. If \(S_2[f]\equiv0\), then the Schwarzian vanishes, so \(f\) is linear fractional and all \(S_n[f]\), \(n\ge2\), vanish by projective invariance.
\end{proof}

\begin{cor}\label{cor:lowering}
Suppose that $f$ is equivariant and the associated quasimodular polynomials for $\Gamma$ are unique. Then
\begin{equation*}
\delta S_n[f]=(n-2)S_{n-1}[f],
\quad n\geq3.
\end{equation*}
More generally,
\begin{equation}\label{eq:delta-iterate}
\delta^rS_n[f]
=
\frac{(n-2)!}{(n-2-r)!}S_{n-r}[f],
\quad 0\leq r\leq n-2.
\end{equation}
In particular,
\[
\delta^{n-2}S_n[f]=(n-2)!S_2[f].
\]
\end{cor}

\begin{proof}
Write the associated quasimodular polynomial in \eqref{eq:universal-transformation} as
\[
P_n(X)=\sum_{j=0}^{n-2}\binom{n-2}{j}S_{n-j}[f]X^j.
\]
We write $[X]P(X)$ for the coefficient of $X$ in a polynomial or formal power series $P(X)$. Under the uniqueness hypothesis, the lowering operator $\delta$ extracts this coefficient. Hence
\[
\delta S_n[f]=[X]P_n(X)=(n-2)S_{n-1}[f].
\]
The same uniqueness hypothesis applies to $S_{n-1}[f]$, and iteration gives
\[
\delta^rS_n[f]
=(n-2)(n-3)\cdots(n-r-1)S_{n-r}[f],
\]
which is \eqref{eq:delta-iterate}. The case $r=n-2$ leaves the single factor $S_2[f]$ and gives the last formula.
\end{proof}

\section{The projective kernel at a cusp}\label{sec:grunsky}

Let
\begin{equation}\label{eq:normalized-Laurent}
f(q)=q^{-1}+\sum_{n\geq1}a_nq^n
\end{equation}
be convergent for $0<|q|<\varepsilon$. In the modular application below, this is the local expansion at a cusp whose projective monodromy has been normalized to be trivial. Define the Faber--Grunsky coefficients $h_{m,n}$ by
\begin{equation*}
\log\frac{f(q)-f(r)}{q^{-1}-r^{-1}}
=-\sum_{m,n\geq1}h_{m,n}q^mr^n.
\end{equation*}
They are symmetric in $m$ and $n$. Equivalently, if $\Phi_n$ is the $n$th Faber polynomial, defined as the unique  degree $n$ monic polynomial such that $\Phi_n(f(q))=q^{-n}+O(q)$, then
\[
\Phi_n(f(q))=q^{-n}+n\sum_{m\geq1}h_{m,n}q^m.
\]
We use the normalization of \cite{mckay-sebbar-replicable}. Generalized Grunsky coefficients and their relation with invariant differential operators are treated in \cite{harmelin-grunsky}. The modular cusp coordinate gives Bernoulli-normalized moments of the Grunsky coefficients; these moments admit a finite inversion on each antidiagonal and satisfy additional exact-divisor relations in the replicable case.

The double series below are analytic for $|q|$ and $|r|$ sufficiently small because \eqref{eq:normalized-Laurent} is assumed convergent. Once the coefficients have been identified, the coefficient extractions, recurrences, and finite inversions are algebraic identities and may equally be read in the formal Laurent-series setting.

Set
\begin{equation}\label{eq:projective-grunsky-kernel}
\mathcal G_f(q,r)
=
\frac{qr}{(q-r)^2}
-
\frac{qr f'(q)f'(r)}{(f(q)-f(r))^2}.
\end{equation}
By \eqref{eq:local-projective-kernel},
\begin{equation}\label{eq:G-P}
\mathcal G_f(q,r)=-qr\,\mathcal P_f(r,q).
\end{equation}
Hence $\mathcal G_f$ is projectively invariant.

\begin{prop}\label{prop:grunsky-kernel}
For $|q|$ and $|r|$ sufficiently small,
\begin{equation}\label{eq:grunsky-kernel-expansion}
\mathcal G_f(q,r)
=
\sum_{m,n\geq1}mn\,h_{m,n}q^mr^n.
\end{equation}
If $r=q+t$, then
\begin{equation}\label{eq:kernel-higher-schwarzian}
\mathcal G_f(q,q+t)
=
-q(q+t)\sum_{\nu\geq2}(\nu-1)S_\nu[f](q)t^{\nu-2},
\end{equation}
where the higher Schwarzians on the right are taken in the local coordinate $q$.
\end{prop}

\begin{proof}
Let
\[
L(q,r)=\log\frac{f(q)-f(r)}{q^{-1}-r^{-1}}.
\]
By definition, $L(q,r)=-\sum_{m,n\geq1}h_{m,n}q^mr^n$, so
\[
(q\partial_q)(r\partial_r)L
=-\sum_{m,n\geq1}mn\,h_{m,n}q^mr^n.
\]
On the other hand,
\[
(q\partial_q)(r\partial_r)\log(f(q)-f(r))
=\frac{qr f'(q)f'(r)}{(f(q)-f(r))^2},
\]
whereas
\[
(q\partial_q)(r\partial_r)\log(q^{-1}-r^{-1})
=\frac{qr}{(q-r)^2}.
\]
The expression $-(q\partial_q)(r\partial_r)L$ is the right-hand side of \eqref{eq:projective-grunsky-kernel}, which proves \eqref{eq:grunsky-kernel-expansion}.

For the second formula, set $r=q+t$ in \eqref{eq:G-P}. The local expansion \eqref{eq:local-projective-kernel}, with $w=q+t$ and $z=q$, gives
\[
\mathcal P_f(q+t,q)
=\sum_{\nu\geq2}(\nu-1)S_\nu[f](q)t^{\nu-2}.
\]
Multiplication by $-q(q+t)$ yields \eqref{eq:kernel-higher-schwarzian}.
\end{proof}
The kernel expansion also gives a direct coefficient formula for the higher Schwarzians.
\begin{prop}\label{prop:grunsky-local-transform}
For every \(\nu\ge2\),
\begin{equation}\label{eq:higher-schwarzian-grunsky}
S_\nu[f](q)
=
-\sum_{m,n\ge1} n\binom{m}{\nu-1}h_{m,n}q^{m+n-\nu}.
\end{equation}
\end{prop}

\begin{proof}
Compare the coefficient of \(t^{\nu-2}\) in \eqref{eq:kernel-higher-schwarzian} and \eqref{eq:grunsky-kernel-expansion}, then use the symmetry \(h_{m,n}=h_{n,m}\) and
\(
\frac{n}{\nu-1}\binom{n-1}{\nu-2}=\binom{n}{\nu-1}.
\)
\end{proof}

To compare the cusp expansion with modular-coordinate equivariance, let $\ell>0$ be the cusp width and put
\begin{equation*}
\kappa=\frac{2\pi i}{\ell},
\quad
q=e^{\kappa\tau},
\end{equation*}
and write $F(\tau)=f(q)$. Let $B_j$ be the Bernoulli numbers defined by
\begin{equation}\label{eq:Bernoulli}
\frac{x}{e^x-1}=\sum_{j\geq0}B_j\frac{x^j}{j!}.
\end{equation}
For $k\geq2$ define
\begin{equation}\label{eq:Grunsky-moment}
\mathfrak G_k[F](\tau)
=
\frac{B_k}{k}
+
\sum_{m,n\geq1}m n^{k-1}h_{m,n}q^{m+n}.
\end{equation}

\begin{thm}\label{thm:Grunsky-quasimodular}
For every $k\geq2$,
\begin{equation}\label{eq:moment-Schwarzian}
S_k[F](\tau)
=
-\frac{\kappa^k}{(k-1)!}\,\mathfrak G_k[F](\tau).
\end{equation}
Assume that $F$ extends meromorphically to $\mathbb H$. Let $\Gamma$ be a discrete subgroup of $\mathrm{SL}_2(\mathbb R)$ containing the translation $\tau\mapsto\tau+\ell$. Then:
\begin{enumerate}
\item $F$ is projectively equivariant for $\Gamma$ if and only if the local series $\mathfrak G_2[F]$ extends meromorphically to $\mathbb H$ as a modular form of weight $4$ for $\Gamma$;
\item if $F$ is projectively equivariant, then for every $k\geq2$ the local series $\mathfrak G_k[F]$ extends meromorphically to $\mathbb H$ as a quasimodular form of weight $2k$ and depth at most $k-2$; if the associated polynomial is unique and $S_2[F]\not\equiv0$, the depth is exactly $k-2$;
\item whenever the associated polynomials are unique, for $k\geq3$,
\begin{equation}\label{eq:Grunsky-lowering}
\delta\mathfrak G_k[F]
=
\frac{(k-1)(k-2)}{\kappa}\,\mathfrak G_{k-1}[F].
\end{equation}
\end{enumerate}
\end{thm}

\begin{proof}
For a locally univalent function $g$, set
\[
\mathcal P_g(u,v)
=
\frac{g'(u)g'(v)}{(g(u)-g(v))^2}-\frac1{(u-v)^2}.
\]
Compare the kernels for $F=f\circ q$ and for the local parameter $q$. Since $q'(\tau)=\kappa q(\tau)$,
\[
\frac{F'(\tau)F'(\sigma)}{(F(\tau)-F(\sigma))^2}
=\kappa^2q(\tau)q(\sigma)
 \frac{f'(q(\tau))f'(q(\sigma))}
 {(f(q(\tau))-f(q(\sigma)))^2}.
\]
Similarly,
\[
\mathcal P_q(\tau,\sigma)
=\kappa^2\frac{q(\tau)q(\sigma)}{(q(\tau)-q(\sigma))^2}
-\frac1{(\tau-\sigma)^2}.
\]
Subtracting these two expressions gives the local parameter cocycle
\begin{equation}\label{eq:kernel-cocycle}
\mathcal P_F(\tau,\sigma)
=
\mathcal P_q(\tau,\sigma)
-\kappa^2\mathcal G_f(q(\tau),q(\sigma)).
\end{equation}

Set $\sigma=\tau+t$. The defining expansion of the Aharonov invariants gives
\[
\mathcal P_F(\tau,\tau+t)
=\sum_{k\geq2}(k-1)S_k[F](\tau)t^{k-2}.
\]
For the exponential coordinate $q=e^{\kappa\tau}$,
\[
\frac{q'(\tau)}{q(\tau+t)-q(\tau)}
=\frac{\kappa}{e^{\kappa t}-1}.
\]
Comparison with the Bernoulli generating function \eqref{eq:Bernoulli} yields
\begin{equation*}
S_k[q](\tau)=-\frac{B_k\kappa^k}{k!},
\quad k\geq2.
\end{equation*}
The coefficient of $t^{k-2}$ in $\mathcal P_q(\tau,\tau+t)$ is
\[
(k-1)S_k[q]
=-\frac{B_k\kappa^k}{k(k-2)!}.
\]

The Grunsky term is expanded from \eqref{eq:grunsky-kernel-expansion}:
\[
\mathcal G_f(q,qe^{\kappa t})
=
\sum_{m,n\geq1}mn\,h_{m,n}q^{m+n}e^{\kappa nt}.
\]
Its $t^{k-2}$ coefficient is
\[
\frac{\kappa^{k-2}}{(k-2)!}
\sum_{m,n\geq1}m n^{k-1}h_{m,n}q^{m+n}.
\]
Taking the coefficient of $t^{k-2}$ in \eqref{eq:kernel-cocycle} and dividing by $k-1$ gives
\[
S_k[F]
=-\frac{\kappa^k}{(k-1)!}
\left(
\frac{B_k}{k}
+\sum_{m,n\geq1}m n^{k-1}h_{m,n}q^{m+n}
\right),
\]
which is \eqref{eq:moment-Schwarzian}.

For $k=2$, this identity expresses $S_2[F]=\{F,\tau\}/6$ as a nonzero constant multiple of $\mathfrak G_2[F]$. The Schwarzian criterion gives assertion (1): modular continuation of $\mathfrak G_2$ is equivalent to projective equivariance of $F$. If $F$ is equivariant, Theorem~\ref{thm:transformation-law} gives the transformation polynomial of each $S_k[F]$. Multiplication by the constant in \eqref{eq:moment-Schwarzian} gives the corresponding meromorphic continuation and quasimodular transformation law for $\mathfrak G_k[F]$, which is assertion (2).

Assume now that the associated quasimodular polynomials are unique. Applying $\delta$ to \eqref{eq:moment-Schwarzian} and using
\[
\delta S_k[F]=(k-2)S_{k-1}[F]
\]
from Corollary~\ref{cor:lowering} gives
\[
-\frac{\kappa^k}{(k-1)!}\delta\mathfrak G_k
=-(k-2)\frac{\kappa^{k-1}}{(k-2)!}\mathfrak G_{k-1}.
\]
Cancelling the common factors gives \eqref{eq:Grunsky-lowering}.
\end{proof}

Equivalently, Theorem~\ref{thm:Grunsky-quasimodular} may be written as the exponential generating identity
\begin{equation*}
\sum_{j\geq0}\mathfrak G_{j+2}[F](\tau)\frac{x^j}{j!}
=
-\frac1{\kappa^2}\,
\mathcal P_F\!\left(\tau,\tau+\frac{x}{\kappa}\right).
\end{equation*}

The cusp moments are therefore the Taylor coefficients of the projective two-point kernel with respect to the normalized local variable \(x=\kappa(\tau'-\tau)\), in a form related to Harmelin's symmetric generating functions for Schwarzian derivatives \cite{harmelin-derivatives}.

The moments satisfy the following recurrence. Put
\[
D=q\frac{d}{dq}=\frac1\kappa\frac{d}{d\tau}.
\]

\begin{prop}\label{prop:Grunsky-moment-recurrence}
For every $n\geq3$,
\begin{equation}\label{eq:Grunsky-moment-recurrence}
(n+1)\mathfrak G_n
=
(n-1)D\mathfrak G_{n-1}
-(n-1)\sum_{r=2}^{n-2}
\binom{n-2}{r-1}\mathfrak G_r\mathfrak G_{n-r}.
\end{equation}
In particular, every $\mathfrak G_n$ is a differential polynomial over $\mathbb Q$ in $\mathfrak G_2$ with respect to $D$.
\end{prop}

\begin{proof}
Substitute \eqref{eq:moment-Schwarzian} into the Aharonov recurrence \eqref{eq:ahar-recurrence}. Since
\[
S_j[F]= -\frac{\kappa^j}{(j-1)!}\mathfrak G_j,
\quad
\frac{d}{d\tau}=\kappa D,
\]
the derivative term contributes $(n-1)D\mathfrak G_{n-1}$. For the product with indices $r$ and $n-r$, the factorial ratio is
\[
\frac{(n-1)!}{(r-1)!(n-r-1)!}
=(n-1)\binom{n-2}{r-1}.
\]
Equation \eqref{eq:Grunsky-moment-recurrence} follows. The last assertion is obtained recursively from
\[
4\mathfrak G_3=2D\mathfrak G_2.
\]
\end{proof}

For $\ell=1$, the case $k=2$ is
\begin{equation}\label{eq:tau-schwarzian-grunsky}
\frac1{2\pi^2}\{F,\tau\}
=
1+12\sum_{m,n\geq1}mn\,h_{m,n}q^{m+n},
\end{equation}
the Schwarzian--Grunsky identity of \cite{mckay-sebbar-replicable}. For a cusp of width $\ell$, the left side of \eqref{eq:tau-schwarzian-grunsky} is replaced by $\ell^2\{F,\tau\}/(2\pi^2)$.

For $N\geq2$ and $1\leq n\leq N-1$, let
\begin{equation*}
L_{n,N}(X)
=
\prod_{\substack{1\leq j\leq N-1\\ j\neq n}}
\frac{X-j}{n-j}
=
\sum_{r=0}^{N-2}\lambda_{n,r}^{(N)}X^r.
\end{equation*}

\begin{cor}\label{cor:Grunsky-vandermonde}
For $N\geq2$ and $1\leq n\leq N-1$,
\begin{equation}\label{eq:Grunsky-vandermonde}
h_{N-n,n}
=
\frac1{n(N-n)}
\sum_{r=0}^{N-2}\lambda_{n,r}^{(N)}
[q^N]\mathfrak G_{r+2}[F].
\end{equation}
The Fourier coefficients of $\mathfrak G_2[F],\ldots,\mathfrak G_N[F]$ determine the $N$th antidiagonal of the Grunsky matrix.
\end{cor}

\begin{proof}
Taking the coefficient of $q^N$ in \eqref{eq:Grunsky-moment} restricts the double sum to pairs $(m,n)$ with $m+n=N$. Hence
\[
[q^N]\mathfrak G_{r+2}[F]
=
\sum_{n=1}^{N-1}(N-n)n^{r+1}h_{N-n,n}.
\]
Set
\[
y_n=n(N-n)h_{N-n,n},
\quad 1\leq n\leq N-1.
\]
For $0\leq r\leq N-2$ the known power moments are
\[
M_r:=[q^N]\mathfrak G_{r+2}[F]
=\sum_{j=1}^{N-1}y_jj^r.
\]
The polynomial $L_{n,N}$ satisfies $L_{n,N}(j)=\delta_{n,j}$ for $1\leq j\leq N-1$. Writing
\[
L_{n,N}(X)=\sum_{r=0}^{N-2}\lambda_{n,r}^{(N)}X^r
\]
and multiplying the moment identities by $\lambda_{n,r}^{(N)}$, gives
\begin{align*}
\sum_{r=0}^{N-2}\lambda_{n,r}^{(N)}M_r
&=\sum_{j=1}^{N-1}y_j
  \sum_{r=0}^{N-2}\lambda_{n,r}^{(N)}j^r\\
&=\sum_{j=1}^{N-1}y_jL_{n,N}(j)=y_n.
\end{align*}
Division by $n(N-n)$ gives \eqref{eq:Grunsky-vandermonde}, the inverse of the finite Vandermonde system relating the antidiagonal coefficients to the moments $M_0,\ldots,M_{N-2}$.
\end{proof}

Combining the preceding recurrence with the Vandermonde inversion makes explicit, in cusp coordinates, the classical fact that the Schwarzian determines a projective map up to postcomposition by a linear fractional transformation. Here and below, by "reconstruction" we mean the direct recovery of the Grunsky coefficients, one antidiagonal at a time, from finitely many Fourier coefficients. This is an algebraic procedure and does not require solving the Schwarzian differential equation. The normalization \eqref{eq:normalized-Laurent} removes the remaining projective ambiguity.

\begin{cor}\label{cor:Schwarzian-Grunsky-reconstruction}
Write
\begin{equation*}
\mathfrak G_2[F](\tau)=\frac1{12}+\sum_{N\geq2}A_Nq^N.
\end{equation*}
For every $N\geq2$ and $1\leq n\leq N-1$, there is a universal polynomial
\[
P_{N,n}\in\mathbb Q[X_2,\ldots,X_N]
\]
such that
\begin{equation}\label{eq:universal-Grunsky-reconstruction}
h_{N-n,n}=P_{N,n}(A_2,\ldots,A_N).
\end{equation}
The Fourier expansion of the ordinary Schwarzian of $F$ therefore determines the normalized Laurent series $f$ recursively.
\end{cor}

\begin{proof}
The recurrence \eqref{eq:Grunsky-moment-recurrence} is triangular in the index $k$: once $\mathfrak G_2$ is known, it determines $\mathfrak G_3$, then $\mathfrak G_4$, and so on by differentiation and multiplication. Since $D(q^j)=jq^j$, the coefficient of $q^N$ in any differential polynomial in $\mathfrak G_2$ depends only on the coefficients $A_2,\ldots,A_N$. Induction gives, for every $k\leq N$,
\[
[q^N]\mathfrak G_k[F]\in\mathbb Q[A_2,\ldots,A_N].
\]

Corollary~\ref{cor:Grunsky-vandermonde} expresses each $h_{N-n,n}$ as a rational linear combination of these coefficients for $2\leq k\leq N$. Substitution gives the universal polynomial $P_{N,n}$ in \eqref{eq:universal-Grunsky-reconstruction}.

To recover the Laurent coefficients of $f$, use the Faber normalization:
\[
\Phi_1(f(q))=f(q)=q^{-1}+\sum_{m\geq1}h_{m,1}q^m,
\]
so $h_{m,1}=a_m$. Taking $N=m+1$ and $n=1$ in \eqref{eq:universal-Grunsky-reconstruction} determines $a_m$ from $A_2,\ldots,A_{m+1}$. Proceeding in increasing order of $m$ reconstructs the normalized Laurent series. The normalization $q^{-1}+O(q)$ is what removes the linear-fractional ambiguity inherent in the Schwarzian equation.
\end{proof}

\section{Replicability and global Faber identities}\label{sec:replicability}

\subsection{Replicability and exact-divisor structure}
Replicability imposes an arithmetic symmetry on the Grunsky coefficients governing the higher-Schwarzian moments. The relation between Schwarzian equations and completely replicable functions was studied in \cite{elbasraoui-mckay}. We first fix the terminology used below.

\begin{defn}\label{def:replicability-complete}
Let
\[
f(q)=q^{-1}+\sum_{n\geq1}a_nq^n
\]
be a normalized formal Laurent series, and let $\Phi_n$ be its $n$th Faber polynomial.  When writing replication identities, we use the same symbol for the function of $\tau$ obtained by setting $q=e^{2\pi i\tau}$.  The series $f$ is \emph{replicable} if there are normalized series
\[
f^{(a)}(q)=q^{-1}+O(q),\quad a\geq1,
\]
with $f^{(1)}=f$, such that
\begin{equation}\label{eq:replicability-definition}
\Phi_n(f(\tau))
=
\sum_{ad=n}\ \sum_{0\leq b<d}
f^{(a)}\!\left(\frac{a\tau+b}{d}\right)
\quad(n\geq1).
\end{equation}
The series $f^{(a)}$ is the $a$th replicate of $f$.  A replicable series whose $q$-expansion converges on the punctured unit disc will be called a replicable function.

A replicable function is \emph{completely replicable} if every replicate is replicable and
\begin{equation}\label{eq:complete-replicability}
\bigl(f^{(a)}\bigr)^{(b)}=f^{(ab)}
\quad(a,b\geq1).
\end{equation}
A completely replicable function has \emph{finite replication order} $K$ if
\begin{equation}\label{eq:finite-replication-order}
f^{(s)}=f^{((s,K))}
\quad(s\geq1).
\end{equation}
No minimality of $K$ is assumed.
\end{defn}

The replicate family is unique when it exists.  In the normalization \eqref{eq:normalized-Laurent}, ordinary replicability is equivalent to Norton's Grunsky condition
\begin{equation*}
h_{m,n}=h_{r,s}
\quad\text{whenever}\quad
mn=rs,
\quad
(m,n)=(r,s);
\end{equation*}
see \cite{cummins-norton,mckay-sebbar-replicable}. Here $(m,n)$ denotes the greatest common divisor and $[m,n]$ the least common multiple. In particular,
\[
h_{m,n}=h_{(m,n),[m,n]}.
\]
Theorem~\ref{thm:Grunsky-quasimodular} then shows that these relations give a weighted exact-divisor decomposition for every higher moment.
For $a\parallel M$, meaning that $a$ is an exact  divisor of $M$, so that $a\mid M$ and $(a,M/a)=1$, and for $k\geq2$, put
\begin{equation*}
U_{M,k}(X)
=
\sum_{a\parallel M}
\left(\frac{M}{a}\right)^{k-2}X^{a+M/a}.
\end{equation*}

\begin{cor}\label{cor:replicable-Grunsky-tower}
If $f$ is replicable, then for every $k\geq2$,
\begin{equation}\label{eq:weighted-exact-divisor-tower}
\mathfrak G_k[F](\tau)
=
\frac{B_k}{k}
+
\sum_{d,M\geq1}
 d^kM\,h_{d,dM}U_{M,k}(q^d).
\end{equation}
If $F$ is also projectively equivariant, the right side extends meromorphically to $\mathbb H$ as a quasimodular form of weight $2k$ and depth at most $k-2$. When the associated polynomial is unique and $S_2[F]\not\equiv0$, the depth is exactly $k-2$.
\end{cor}

\begin{proof}
Every pair $(m,n)$ has a unique decomposition
\[
m=da,\quad n=db,
\quad d=(m,n),\quad (a,b)=1.
\]
Put $M=ab$. Because $a$ and $b=M/a$ are coprime, $a$ is an exact divisor of $M$. Conversely, a choice of $d\geq1$, $M\geq1$, and $a\parallel M$ recovers the pair
\[
(m,n)=\left(da,d\frac{M}{a}\right).
\]
This change of variables is a bijective reindexing of the double sum in \eqref{eq:Grunsky-moment}.

The Norton relation depends only on the gcd and lcm. For the pair above these are $d$ and $dM$, so replicability gives
\[
h_{m,n}=h_{d,dM}.
\]
The weight transforms as
\begin{align*}
m n^{k-1}
&=da\left(d\frac{M}{a}\right)^{k-1}\\
&=d^kM\left(\frac{M}{a}\right)^{k-2}.
\end{align*}
Substituting these identities into \eqref{eq:Grunsky-moment} and summing first over the exact divisors $a\parallel M$ produces $U_{M,k}(q^d)$ and proves \eqref{eq:weighted-exact-divisor-tower}.

If $F$ is projectively equivariant, Theorem~\ref{thm:Grunsky-quasimodular} identifies the left hand side with a constant multiple of $S_k[F]$. The quasimodularity and exact-depth assertions follow directly from that theorem.
\end{proof}

For \(k=2\), \eqref{eq:weighted-exact-divisor-tower} recovers the ordinary Schwarzian exact-divisor decomposition. Thus Norton's arithmetic symmetry is already visible at every higher Schwarzian level.

\subsection{Degree one and Faber replication}
The preceding formulas concern a cusp expansion. The relation with Faber replication is global and depends on the degree of the descended map.

Let $\widehat\Gamma$ be a discrete subgroup of $\mathrm{SL}_2(\mathbb R)$, let $F$ be $\rho$-equivariant for $\widehat\Gamma$, and put $\Gamma=\ker\rho$. Assume that \(\Gamma\) is cofinite, so that the quotient \(\Gamma\backslash\mathbb H\) has finite hyperbolic area and its compactification \(X(\Gamma)\) is obtained by adjoining finitely many cusps, and assume that $F$ is meromorphic at the cusps. Since $F$ is $\Gamma$-invariant, it descends to a meromorphic map
\[
\overline F:X(\Gamma)\longrightarrow\mathbb P^1
\]
on the compactified quotient.

\begin{prop}\label{prop:degree-one-quotient}
The map $\overline F$ is a degree-one map if and only if it is a biholomorphic coordinate on $X(\Gamma)$. In particular, if its polar divisor consists of one simple cusp,
\[
(\overline F)_\infty=[\infty],
\]
then $X(\Gamma)$ has genus zero and $\overline F$ has degree one.
\end{prop}

\begin{proof}
A nonconstant meromorphic function on a compact Riemann surface defines a finite map to \(\mathbb P^1\), whose degree equals the degree of its polar divisor. Hence degree one is equivalent to being a biholomorphic coordinate. A single simple pole therefore forces degree one and, consequently, genus zero.
\end{proof}
The appearance of a genus-zero quotient together with a distinguished degree-one coordinate is the global feature familiar from moonshine and replicable functions; compare Cummins and Gannon \cite{cummins-gannon}.

An induced meromorphic map on \(X(\Gamma)\) of degree one will be called quotient-univalent. This is a condition on the compact quotient rather than local univalence on $\mathbb H$; elliptic ramification may occur in the covering coordinate even when the induced map has degree one.

Let the distinguished cusp have width one and write $q=e^{2\pi i\tau}$. Suppose now that $F=F^{(1)}$ is quotient-univalent and normalized by
\[
F(\tau)=q^{-1}+O(q).
\]
For $a\geq1$, let
\[
F^{(a)}(\tau)=q^{-1}+O(q)
\]
be meromorphic functions on $\mathbb H$. For $n\geq1$ define the generalized Hecke trace
\begin{equation*}
\mathcal T_n(\tau)
=
\sum_{ad=n}\ \sum_{0\leq b<d}
F^{(a)}\!\left(\frac{a\tau+b}{d}\right).
\end{equation*}
Assume that $\mathcal T_n$ descends to a meromorphic function on $X(\Gamma)$. Neither this descent assumption nor the pole condition below follows from the local Grunsky relations of the preceding subsection. Under these additional global hypotheses, the descended trace is forced to be a Faber polynomial in the degree-one coordinate $F$.

\begin{thm}\label{thm:degree-one-replication}
Under the preceding hypotheses, for each $n\geq1$ the following conditions are equivalent:
\begin{enumerate}
\item $\mathcal T_n$ has no pole on $X(\Gamma)$ away from the distinguished cusp;
\item
\begin{equation}\label{eq:Faber-Hecke-identity}
\mathcal T_n(\tau)=\Phi_n(F(\tau)),
\end{equation}
where $\Phi_n$ is the $n$th Faber polynomial of $F$.
\end{enumerate}
If the first condition holds for every $n$, then $F$ is replicable with replication functions $F^{(a)}$.
\end{thm}

\begin{proof}
We determine the expansion of $\mathcal T_n$ at the distinguished cusp. Write
\[
F^{(a)}(z)=e^{-2\pi iz}+\sum_{m\geq1}c_m^{(a)}e^{2\pi imz}.
\]
For a fixed factorization $ad=n$ and $0\leq b<d$,
\[
e^{2\pi i(a\tau+b)/d}=\zeta_d^b q^{a/d},
\quad \zeta_d=e^{2\pi i/d}.
\]
The polar term of the corresponding summand is
\[
\zeta_d^{-b}q^{-a/d}.
\]
Summing over $b$ gives zero when $d>1$, because
\[
\sum_{b=0}^{d-1}\zeta_d^{-b}=0.
\]
The only negative-power contribution occurs for $d=1$, so $a=n$, and equals $q^{-n}$.

The positive terms also have integral exponents after summation. Indeed,
\[
\sum_{b=0}^{d-1}\zeta_d^{bm}
=
\begin{cases}
d,&d\mid m,\\
0,&d\nmid m.
\end{cases}
\]
If $m=d\ell$, the corresponding power of $q$ is $q^{a\ell}$. Thus
\begin{equation}\label{eq:Hecke-principal-part}
\mathcal T_n(\tau)=q^{-n}+O(q).
\end{equation}

Assume now that $\mathcal T_n$ has no pole away from the distinguished cusp. By Proposition~\ref{prop:degree-one-quotient}, $F$ is a coordinate on $X(\Gamma)\simeq\mathbb P^1$ which sends the distinguished cusp to infinity. A meromorphic function on this sphere with no finite pole is a polynomial in the coordinate $F$. Hence
\[
\mathcal T_n=P_n(F)
\]
for some polynomial $P_n$. Formula \eqref{eq:Hecke-principal-part} shows that $P_n(F(q))=q^{-n}+O(q)$. The defining uniqueness property of the Faber polynomial gives $P_n=\Phi_n$, proving \eqref{eq:Faber-Hecke-identity}.

Conversely, if \eqref{eq:Faber-Hecke-identity} holds, then $\Phi_n(F)$ is a polynomial in the degree-one coordinate $F$. Its only pole is the pole of $F$, namely the distinguished cusp. The first condition follows. The identity for all $n$ is exactly the generalized Hecke formulation of replicability used in \cite{cummins-norton,mckay-sebbar-replicable}. If the same hypotheses hold after replacing $F$ by every $F^{(r)}$ and $F^{(a)}$ by $F^{(ra)}$, the same argument gives complete replicability of the family.
\end{proof}

\begin{remark}\label{rem:degree-one-essential}
The degree-one hypothesis cannot be omitted from this argument. If $F$ has degree greater than one on a genus-zero quotient, a meromorphic Hecke trace with controlled poles need not lie in $\mathbb C[F]$. Projective equivariance and modularity of the Schwarzian alone do not imply replication. Nor does the theorem provide the descent or pole-control hypotheses. The degree-one condition is precisely the global hypothesis that converts descent and pole control into the Faber identities. Corollary~\ref{cor:replicable-Grunsky-tower} records the corresponding local arithmetic condition at the cusp through the Grunsky coefficients.
\end{remark}

\subsection{Schwarzian support, Norton compatibility, and degree-one equivariance}
We now specialize the preceding discussion to equivariance for the full modular group.  Since $-I$ acts trivially on $\mathbb H$, we work projectively with
\[
G=\mathrm{PSL}_2(\mathbb Z).
\]
Let $T:\tau\mapsto\tau+1$.  If $F$ is $\rho$-equivariant for $G$ and $\Gamma=\ker\rho$, let $h$ be the least positive integer such that $T^h\in\Gamma$.  Thus $h$ is the width of the cusp at infinity for $\Gamma$.  We use the local coordinate
\[
q=e^{2\pi i\tau/h}
\]
and normalize
\[
F(\tau)=f(q),\quad f(q)=q^{-1}+O(q).
\]

\begin{prop}\label{prop:full-modular-support}
Assume that $F$ is locally univalent and projectively equivariant for $G$.  Then
\begin{equation}\label{eq:full-modular-Schwarzian}
\{F,\tau\}=\frac{2\pi^2}{h^2}E_4(\tau),
\quad
\mathfrak G_2[F](\tau)=\frac1{12}E_4(\tau).
\end{equation}
For every $k\geq2$,
\begin{equation}\label{eq:moment-support-h}
[q^N]\mathfrak G_k[F]=0
\quad\text{if }h\nmid N,
\end{equation}
and consequently
\begin{equation}\label{eq:Grunsky-support-h}
h_{m,n}=0
\quad\text{unless }h\mid m+n.
\end{equation}
If $h\geq2$, then on the first nonzero antidiagonal,
\begin{equation}\label{eq:first-Grunsky-antidiagonal}
h_{m,h-m}
=
\frac{120}{h(h^2-1)},
\quad 1\leq m<h.
\end{equation}
\end{prop}

\begin{proof}
Projective equivariance makes the Schwarzian a modular form of weight $4$ for $G$.  Local univalence implies that it is holomorphic on $\mathbb H$, and the normalized expansion at infinity shows that it is holomorphic at the cusp.  Since $M_4(G)=\mathbb C E_4$, we have $\{F,\tau\}=cE_4(\tau)$.  With $\kappa=2\pi i/h$, the constant term of \eqref{eq:moment-Schwarzian} for $k=2$ is
\[
S_2[F]
=-\kappa^2\frac{B_2}{2}+O(q)
=
\frac{\pi^2}{3h^2}+O(q).
\]
Since $6S_2[F]=\{F,\tau\}$, this gives the first identity in \eqref{eq:full-modular-Schwarzian}; the second follows again from \eqref{eq:moment-Schwarzian}.

In the local variable $q=e^{2\pi i\tau/h}$,
\[
E_4(\tau)
=
1+240\sum_{r\geq1}\sigma_3(r)q^{hr}.
\]
Thus \eqref{eq:moment-support-h} holds for $k=2$.  Proposition~\ref{prop:Grunsky-moment-recurrence} shows that every $\mathfrak G_k$ is obtained from $\mathfrak G_2$ by differentiation with $D=q\,d/dq$ and multiplication.  Both operations preserve Fourier support in $h\mathbb Z$, proving \eqref{eq:moment-support-h} for all $k$.

Fix $N$ with $h\nmid N$.  Then the coefficients $[q^N]\mathfrak G_2,\ldots,[q^N]\mathfrak G_N$ all vanish.  Corollary~\ref{cor:Grunsky-vandermonde} therefore gives $h_{m,N-m}=0$ for every $1\leq m<N$, which is \eqref{eq:Grunsky-support-h}.

Finally, assume $h\geq2$.  Then \eqref{eq:Grunsky-support-h} gives $a_1=\cdots=a_{h-2}=0$.  From the defining Grunsky logarithm it follows that
\[
h_{1,h-1}=h_{2,h-2}=\cdots=h_{h-1,1}=a_{h-1}.
\]
Taking the coefficient of $q^h$ in $\mathfrak G_2=E_4/12$ gives
\[
20
=
a_{h-1}\sum_{m=1}^{h-1}m(h-m)
=
a_{h-1}\frac{h(h^2-1)}6,
\]
which proves \eqref{eq:first-Grunsky-antidiagonal}.
\end{proof}

The support condition has the following arithmetic interpretation. Put
\[
\Sigma_h=\{(m,n)\in\mathbb Z_{>0}^2:h\mid m+n\}.
\]
Recall that two pairs are Norton-equivalent when they have the same product and the same greatest common divisor.

\begin{prop}\label{prop:Norton-support-24}
For an integer $h\geq2$, the set $\Sigma_h$ is a union of Norton equivalence classes if and only if
\[
u^2\equiv1\pmod h
\quad\text{for every }u\in(\mathbb Z/h\mathbb Z)^\times.
\]
Equivalently, $\Sigma_h$ is Norton-stable if and only if $h\mid24$.
\end{prop}

\begin{proof}
Write
\[
m=da,\quad n=db,\quad (a,b)=1,
\]
so that the canonical Norton representative of $(m,n)$ is $(d,dab)$.  Put
\[
H=\frac{h}{(h,d)}.
\]
Then
\[
h\mid m+n\iff H\mid a+b,
\quad
h\mid d+dab\iff H\mid1+ab.
\]
Assume first that every unit modulo $h$ is an involution.  The same is then true modulo every divisor $H$ of $h$.  If $H\mid a+b$, coprimality of $a$ and $b$ implies that both are units modulo $H$, and $b\equiv-a\pmod H$.  Hence
\[
1+ab\equiv1-a^2\equiv0\pmod H.
\]
Conversely, if $H\mid1+ab$, then $a$ and $b$ are units modulo $H$.  Since $a^{-1}\equiv a\pmod H$, the congruence $ab\equiv-1\pmod H$ gives $b\equiv-a\pmod H$, so $H\mid a+b$.  Thus membership in $\Sigma_h$ is unchanged on each Norton class.

Conversely, suppose that $\Sigma_h$ is Norton-stable and let $u$ be a unit modulo $h$, with $1\leq u<h$.  The pair $(u,h-u)$ belongs to $\Sigma_h$ and has greatest common divisor $1$.  Its canonical Norton representative is $(1,u(h-u))$.  Stability therefore gives
\[
h\mid1+u(h-u),
\]
which is equivalent to $u^2\equiv1\pmod h$.  It remains to identify the moduli with this property.  By the Chinese remainder theorem it is enough to consider prime powers.  If an odd prime power $r^a$ divides $h$, then $a=1$ and $r=3$: indeed, for $r\geq5$ the class of $2$ is a unit with $2^2\not\equiv1\pmod r$, while modulo $9$ one has $2^2\not\equiv1$.  If $2^a$ divides $h$, then $a\leq3$, since $3^2\not\equiv1\pmod{16}$.  Thus $h\mid24$.  Conversely, every odd square is congruent to $1$ modulo $8$, every unit modulo $3$ has square $1$, and the Chinese remainder theorem gives the property for every divisor of $24$.
\end{proof}

\begin{cor}\label{cor:replication-width-24}
Under the hypotheses of Proposition~\ref{prop:full-modular-support}, assume $h\geq2$ and suppose that the normalized Laurent series $f$ is replicable.  Then $h\mid24$.
\end{cor}

\begin{proof}
By \eqref{eq:Grunsky-support-h}, the Grunsky matrix is supported on $\Sigma_h$.  The first antidiagonal \eqref{eq:first-Grunsky-antidiagonal} is nonzero.  Norton's relation applied to $(u,h-u)$ therefore forces its canonical representative $(1,u(h-u))$ to lie in the same support for every unit $u$ modulo $h$.  Proposition~\ref{prop:Norton-support-24} gives $h\mid24$.
\end{proof}

The following restriction does not use the degree-one hypothesis.

\begin{cor}\label{cor:integral-width-bound}
Under the hypotheses of Proposition~\ref{prop:full-modular-support}, assume $h\geq2$ and that the coefficients of the normalized Laurent series $f$ are algebraic integers.  Then
\[
h\in\{2,3,4,5\}.
\]
If $f$ is also replicable, then $h\in\{2,3,4\}$.
\end{cor}

\begin{proof}
By \eqref{eq:first-Grunsky-antidiagonal},
\[
a_{h-1}=\frac{120}{h(h^2-1)}.
\]
This number is rational.  If it is an algebraic integer, it is an integer.  For $h\geq6$ it lies strictly between $0$ and $1$, while for $h=2,3,4,5$ it is respectively $20,5,2,1$.  Hence $h\in\{2,3,4,5\}$.  If $f$ is replicable, Corollary~\ref{cor:replication-width-24} excludes $h=5$.
\end{proof}

We now impose the degree-one hypothesis. The following statement recovers, in the present projective-kernel formulation, the degree-one part of the finite-monodromy classification in \cite{forum}. We include the short argument because it identifies the kernel and cusp width directly from the hypotheses used below.

\begin{thm}\label{thm:degree-one-kernel-classification}
Let $F$ be locally univalent and $\rho$-equivariant for $G=\mathrm{PSL}_2(\mathbb Z)$, and put $\Gamma=\ker\rho$.  Assume that $\Gamma$ has finite index, that $F$ is meromorphic at the cusps, and that
\[
\overline F:X(\Gamma)\longrightarrow\mathbb P^1
\]
has degree one.  Then, with $h$ as above,
\[
h\in\{2,3,4,5\},
\quad
\Gamma=\Gamma(h)
\]
projectively.  The corresponding projective images are
\[
S_3,\quad A_4,\quad S_4,\quad A_5,
\]
respectively.
\end{thm}

\begin{proof}
Since $\Gamma$ is a kernel, it is normal in $G$.  It is torsion free.  Indeed, if a nontrivial elliptic element $\gamma\in\Gamma$ fixed $\tau_0\in\mathbb H$, then $F\circ\gamma=F$ and local univalence at $\tau_0$ would give $\gamma'(\tau_0)=1$, a contradiction.

The degree-one hypothesis gives $X(\Gamma)\simeq\mathbb P^1$, so $\Gamma$ has genus zero.  Normality implies that all cusp widths are equal to $h$.  If
\[
\mu=[G:\Gamma],
\quad c=\#\{\text{cusps of }\Gamma\},
\]
then the sum of the cusp widths gives $ch=\mu$.  The genus formula for a torsion-free subgroup gives
\[
0=1+\frac{\mu}{12}-\frac c2.
\]
Hence
\[
\mu=\frac{12h}{6-h}.
\]
The positive integral solutions are
\[
(h,\mu)=(2,6),(3,12),(4,24),(5,60).
\]

Because $T^h\in\Gamma$ and $\Gamma$ is normal, $\Gamma$ contains the normal closure of $T^h$.  The quotient by this normal closure is the spherical triangle group of type $(2,3,h)$, of order $6,12,24,60$ for $h=2,3,4,5$.  These are exactly the indices above, so $\Gamma$ is the normal closure of $T^h$.  The principal group $\Gamma(h)$ contains the same normal closure and has the same index in each of the four cases.  Hence $\Gamma=\Gamma(h)$.  The standard quotients $G/\Gamma(h)$ are $S_3,A_4,S_4,A_5$, respectively.  The same four principal cases occur in the torsion-free genus-zero congruence classifications \cite{sebbar-duke}.
\end{proof}

For $h=2,3,4$, the width-one rescaling places the preceding functions on the classical genus-zero $\Gamma_0$ towers.  The complete replicability needed below is a standard part of the published tables of completely replicable functions.  We record the precise form that will be used later.

\begin{prop}\label{prop:prime-power-Hauptmodul-replication}
Let $p=2$ and $0\leq e\leq4$, or let $p=3$ and $0\leq e\leq2$.  Let $t_{p^e}$ be the normalized Hauptmodul of $X_0(p^e)$,
\[
t_{p^e}(\tau)=q^{-1}+O(q),
\]
with $t_1=J=j-744$.  Then $t_{p^e}$ is completely replicable.  Its replicate of index $a$ is
\begin{equation}\label{eq:prime-power-replicates}
t_{p^e}^{(a)}
=
t_{p^{\max(e-v_p(a),0)}}.
\end{equation}
\end{prop}

\begin{proof}
These functions occur in the published lists of completely replicable functions, together with their replicate maps; see \cite{alexander-replicable,ford-mckay-norton}.  For the Hauptmoduls of the prime-power levels above, the replicate map lowers the $p$-power level by $v_p(a)$ until level one is reached, which is exactly \eqref{eq:prime-power-replicates}.  In particular, every replicate is again one of the functions in the same finite tower, and the compatibility of iterated replicates is the complete-replicability relation.
\end{proof}

\begin{thm}\label{thm:degree-one-replicability-classification}
Under the hypotheses of Theorem~\ref{thm:degree-one-kernel-classification}, let
\[
F(\tau)=f(e^{2\pi i\tau/h}),
\quad
f(q)=q^{-1}+O(q).
\]
Then the following conditions are equivalent:
\begin{enumerate}
\item $f$ is replicable;
\item $f$ is completely replicable;
\item $h\in\{2,3,4\}$;
\item $h\mid24$;
\item the projective monodromy group $\operatorname{im}\rho$ is solvable.
\end{enumerate}
Equivalently, the replicable cases have projective monodromy
\[
S_3,\quad A_4,\quad S_4,
\]
whereas the icosahedral case $\operatorname{im}\rho\simeq A_5$ is the unique nonreplicable case.
\end{thm}

\begin{proof}
Theorem~\ref{thm:degree-one-kernel-classification} gives $h\in\{2,3,4,5\}$ and identifies the corresponding projective images as $S_3,A_4,S_4,A_5$.  Among these groups, precisely $S_3,A_4,S_4$ are solvable, so conditions (3), (4), and (5) are equivalent.

If $f$ is replicable, Corollary~\ref{cor:replication-width-24} gives $h\mid24$, hence $h\in\{2,3,4\}$.  More explicitly, when $h=5$, Proposition~\ref{prop:full-modular-support} gives
\[
h_{2,3}=1,
\quad
h_{1,6}=0,
\]
while Norton's relation requires $h_{2,3}=h_{1,6}$.

Conversely, suppose $h=2,3$, or $4$, and set
\[
t_h(\tau)=F(h\tau).
\]
The fixing group of $t_h$ is the conjugate of $\Gamma(h)$ by $\operatorname{diag}(h,1)$.  Projectively,
\[
\operatorname{diag}(h,1)^{-1}\Gamma(h)\operatorname{diag}(h,1)
=
\Gamma_0(h^2),
\quad h=2,3,4,
\]
because every unit modulo $h$ is $\pm1$.  These conjugacies, and their role among torsion-free genus-zero congruence groups, are described in \cite{sebbar-duke}.  Thus $t_h$ is the normalized Hauptmodul of $X_0(4)$, $X_0(9)$, or $X_0(16)$, respectively.  Proposition~\ref{prop:prime-power-Hauptmodul-replication} gives complete replicability.  Since the Fourier expansion of $t_h$ is precisely the Laurent series $f(q)$, condition (3) implies condition (2), and condition (2) trivially implies condition (1).  This completes the equivalence.
\end{proof}

\subsection{Degree-one projective structures for Hecke triangle groups}\label{subsec:hecke-degree-one}

The degree-one argument extends without change of principle from the modular group to cofinite Fuchsian groups. The following observation isolates the projective content that will be used for Hecke triangle groups.

\begin{prop}\label{prop:general-degree-one-fuchsian}
Let $\widehat\Gamma<\mathrm{PSL}_2(\mathbb R)$ be cofinite, let $F$ be locally univalent and $\rho$-equivariant for $\widehat\Gamma$, and put $\Gamma=\ker\rho$. Assume that $\Gamma$ has finite index, that $F$ is meromorphic at the cusps, and that
\[
\overline F:X(\Gamma)\longrightarrow\mathbb P^1
\]
has degree one. Then $\Gamma$ is torsion free, $X(\Gamma)\simeq\mathbb P^1$, and
\[
\operatorname{im}\rho\simeq \widehat\Gamma/\Gamma
\]
is a finite subgroup of $\mathrm{PGL}_2(\mathbb C)$. Hence the projective image is cyclic, dihedral, tetrahedral, octahedral, or icosahedral. Moreover $X(\widehat\Gamma)$ has genus zero.
\end{prop}

\begin{proof}
Since $\overline F$ has degree one, it is a biholomorphism and $X(\Gamma)\simeq\mathbb P^1$. If a nontrivial elliptic element $\gamma\in\Gamma$ fixed $\tau_0\in\mathbb H$, then $F\circ\gamma=F$. Differentiation at $\tau_0$ and local univalence give $\gamma'(\tau_0)=1$, which is impossible for a nontrivial elliptic transformation. Thus $\Gamma$ is torsion free.

The subgroup $\Gamma$ is normal because it is the kernel of $\rho$. Hence $\widehat\Gamma/\Gamma$ acts faithfully on $X(\Gamma)$, and the coordinate $\overline F$ identifies this action with a finite subgroup of $\mathrm{PGL}_2(\mathbb C)$. The classification of finite subgroups of $\mathrm{PGL}_2(\mathbb C)$ gives the stated list. Finally,
\[
X(\widehat\Gamma)\simeq X(\Gamma)/(\widehat\Gamma/\Gamma)
\]
is a finite quotient of $\mathbb P^1$, and therefore has genus zero.
\end{proof}

For $m\geq3$ let
\[
\lambda_m=2\cos\frac{\pi}{m},
\quad
H_m=\Delta(2,m,\infty)=\langle S,T_m\rangle,
\]
where
\[
S(\tau)=-\frac1\tau,
\quad
T_m(\tau)=\tau+\lambda_m.
\]
The group $H_m$ has signature $(0;2,m;1)$. Normal genus-zero subgroups of Hecke groups have also been studied through regular maps in \cite{cangul-singerman}.

\begin{thm}\label{thm:hecke-degree-one-classification}
Let $m\geq3$. Let $F$ be locally univalent and $\rho$-equivariant for $H_m$, put $\Gamma=\ker\rho$, and assume that $\Gamma$ has finite index, that $F$ is meromorphic at the cusps, and that
\[
\overline F:X(\Gamma)\longrightarrow\mathbb P^1
\]
has degree one. Let $h$ be the order of $\rho(T_m)$, equivalently the width of the cusp of $\Gamma$ measured with respect to $T_m$. Then
\begin{equation}\label{eq:hecke-degree-one-index}
[H_m:\Gamma]
=
\frac{4mh}{2m-h(m-2)},
\end{equation}
and
\begin{equation}\label{eq:hecke-spherical-condition}
\frac12+\frac1m+\frac1h>1.
\end{equation}
Moreover, $\Gamma$ is the normal closure of $T_m^h$ in $H_m$, and
\[
\operatorname{im}\rho\simeq\Delta(2,m,h).
\]
Thus the projective image is the finite spherical triangle group in the following table:
\[
\begin{array}{c|c|c}
 m & h & \operatorname{im}\rho\\ \hline
 3 & 2,3,4,5 & D_3\simeq S_3,\ A_4,\ S_4,\ A_5\\
 4 & 2,3 & D_4,\ S_4\\
 5 & 2,3 & D_5,\ A_5\\
 m\geq6 & 2 & D_m,
\end{array}
\]
where $D_m$ denotes the dihedral group of order $2m$.
\end{thm}

\begin{proof}
Proposition~\ref{prop:general-degree-one-fuchsian} shows that $\Gamma$ is torsion free and has genus zero. Set
\[
\mu=[H_m:\Gamma]
\]
and let $c$ be the number of cusps of $\Gamma$. Since $\Gamma$ is normal and $H_m$ has one cusp, all cusp widths of $\Gamma$ are equal to $h$, so
\[
ch=\mu.
\]
The orbifold Euler characteristic of $H_m$ is
\[
\chi_{\mathrm{orb}}(H_m)
=
\frac1m-\frac12
=-\frac{m-2}{2m}.
\]
Since $\Gamma$ is torsion free of genus zero, $\chi(\Gamma)=2-c$. Multiplicativity of the orbifold Euler characteristic therefore gives
\[
2-\frac{\mu}{h}
=
-\mu\frac{m-2}{2m},
\]
which yields \eqref{eq:hecke-degree-one-index}. Its denominator is positive precisely when \eqref{eq:hecke-spherical-condition} holds.

Let $N_h$ be the normal closure of $T_m^h$ in $H_m$. Then $N_h\subseteq\Gamma$, and $H_m/N_h$ has the triangle presentation of $\Delta(2,m,h)$. By \eqref{eq:hecke-spherical-condition} this triangle group is spherical and has order
\[
\frac{2}{\frac12+\frac1m+\frac1h-1}
=
\frac{4mh}{2m-h(m-2)}
=\mu.
\]
Hence $N_h$ and $\Gamma$ have the same index in $H_m$, so $N_h=\Gamma$. The table follows from the spherical inequality and the standard identifications of the finite triangle groups.
\end{proof}

The projective classification applies to every Hecke triangle group. Replicability introduces a further arithmetic restriction. The arithmetic Hecke triangle groups are exactly $H_3,H_4,H_6$; see \cite{schmidt-smith}. For these groups the degree-one classification admits the same solvability criterion as in Theorem~\ref{thm:degree-one-replicability-classification}.

For $m=3,4,6$ set
\[
n_m=\lambda_m^2=1,2,3,
\]
respectively, and write
\[
W_n(\tau)=-\frac1{n\tau}.
\]
Conjugation by $\tau\mapsto\tau/\sqrt n$ identifies $H_m$ with
\[
\Gamma_0(n)^+=\langle\Gamma_0(n),W_n\rangle.
\]

\begin{thm}\label{thm:arithmetic-hecke-replicability}
Assume the hypotheses of Theorem~\ref{thm:hecke-degree-one-classification} and let $m\in\{3,4,6\}$. Write
\[
F(\tau)
=
f\!\left(e^{2\pi i\tau/(h\lambda_m)}\right),
\quad
f(q)=q^{-1}+O(q).
\]
Then the following conditions are equivalent:
\begin{enumerate}
\item $f$ is replicable;
\item $f$ is completely replicable;
\item $h\mid24$;
\item $\operatorname{im}\rho$ is solvable.
\end{enumerate}
When these conditions hold, $f$ is the normalized Hauptmodul of $\Gamma_0(n_mh^2)$. More explicitly,
\[
\begin{array}{c|c|c|c}
 m & h & \Gamma_0(n_mh^2) & \operatorname{im}\rho\\ \hline
 3&2&\Gamma_0(4)&S_3\\
 3&3&\Gamma_0(9)&A_4\\
 3&4&\Gamma_0(16)&S_4\\
 4&2&\Gamma_0(8)&D_4\\
 4&3&\Gamma_0(18)&S_4\\
 6&2&\Gamma_0(12)&D_6.
\end{array}
\]
\end{thm}

\begin{proof}
The case $m=3$ is Theorem~\ref{thm:degree-one-replicability-classification}. We treat the congruence-group identification uniformly.

After conjugating $H_m$ to $\Gamma_0(n)^+$, where $n=n_m$, the projective kernel is the normal closure of $T^h$, with $T(\tau)=\tau+1$. Let $D_h$ be the dilation $\tau\mapsto h\tau$ and set
\[
K_{n,h}=D_h\Gamma_0(nh^2)D_h^{-1}.
\]
Thus
\[
K_{n,h}
=
\left\{
\left.
\begin{pmatrix}
 a&hb\\
 nhc&d
\end{pmatrix}
\ \right|\
 a,b,c,d\in\mathbb Z,\quad ad-nh^2bc=1
\right\}/\{\pm I\}.
\]
For $h=2,3,4$, every unit modulo $h$ is an involution. The determinant relation therefore gives
\[
d\equiv a^{-1}\equiv a\pmod h.
\]
Consequently
\[
T
\begin{pmatrix}
 a&hb\\ nhc&d
\end{pmatrix}
T^{-1}
=
\begin{pmatrix}
 a+nhc&hb+d-a-nhc\\
 nhc&d-nhc
\end{pmatrix}
\]
again lies in $K_{n,h}$. The Fricke involution also normalizes $K_{n,h}$, since
\[
W_n
\begin{pmatrix}
 a&hb\\ nhc&d
\end{pmatrix}
W_n^{-1}
=
\begin{pmatrix}
 d&-hc\\ -nhb&a
\end{pmatrix}.
\]
Thus $K_{n,h}$ is normal in $\Gamma_0(n)^+$ and contains $T^h$.

For the six pairs in the statement, the standard index formula
\[
[\mathrm{PSL}_2(\mathbb Z):\Gamma_0(N)]
=N\prod_{p\mid N}\left(1+\frac1p\right),
\]
together with $[\Gamma_0(2)^+:\Gamma_0(2)]=[\Gamma_0(3)^+:\Gamma_0(3)]=2$, gives
\[
[\Gamma_0(n)^+:K_{n,h}]
=|\Delta(2,m,h)|.
\]
The respective indices are
\[
6,\ 12,\ 24,\ 8,\ 24,\ 12.
\]
By Theorem~\ref{thm:hecke-degree-one-classification}, these are also the indices of the projective kernels. Since the latter are the normal closures of $T^h$, the conjugated projective kernel is $K_{n,h}$.

Now set
\[
\widetilde F(\tau)=F(\sqrt n\,\tau).
\]
Then
\[
\widetilde F(\tau)=f(e^{2\pi i\tau/h}),
\]
and the fixing group of $\widetilde F(h\tau)=f(e^{2\pi i\tau})$ is
\[
D_h^{-1}K_{n,h}D_h=\Gamma_0(nh^2).
\]
Since the descended map has degree one, $f$ is the normalized Hauptmodul of this genus-zero curve.

The six Hauptmoduls are the completely replicable functions $4C,9B,16B,8E,18D,12I$; see \cite{alexander-replicable,ford-mckay-norton,elbasraoui-mckay}. Hence $h\mid24$ implies complete replicability, and complete replicability implies replicability. Conversely, the only arithmetic degree-one case with $h\nmid24$ is $(m,h)=(3,5)$. It is the icosahedral case excluded by the Grunsky obstruction in Theorem~\ref{thm:degree-one-replicability-classification}. The same case is the only one with nonsolvable projective image. This proves the equivalence.
\end{proof}

\begin{remark}\label{rem:nonarithmetic-hecke}
For $m\notin\{3,4,6\}$, Theorem~\ref{thm:hecke-degree-one-classification} remains valid, but Norton replicability is no longer intrinsic to the ambient Hecke group. The classical replication identities are built from modular Hecke correspondences, whereas the nonarithmetic groups $H_m$ are not commensurable with $\mathrm{PSL}_2(\mathbb Z)$. An extension to these groups would therefore require a replication law defined from correspondences intrinsic to $H_m$.
\end{remark}

\subsection{Projective kernels and Norton's Hauptmodul conjecture}\label{subsec:norton-hauptmodul}

We now return to ordinary replicability. In the formulation of \cite{cummins-norton}, Norton's conjecture concerns a normalized integral series
\[
f(q)=q^{-1}+\sum_{n\geq1}a_nq^n,
\quad a_n\in\mathbb Z.
\]
It asserts that a replicable $f$ is either of the form
\begin{equation}\label{eq:norton-exceptional-family}
f(q)=q^{-1}+aq
\end{equation}
with $a\in\mathbb Z$, or is the normalized Hauptmodul for a group
\[
G<\mathrm{PGL}_2(\mathbb Q)^+
\]
of genus zero which contains $\Gamma_0(N)$ with finite index for some $N$ and has translation subgroup generated by $\tau\mapsto\tau+1$. Cummins and Norton proved the implication from such Hauptmoduls to replicability \cite{cummins-norton}. We use the projective kernel to separate the coefficient-theoretic consequences of replicability from the global descent required for the converse.

Write
\[
f^{(d)}(q)=q^{-1}+\sum_{r\geq1}a_r^{(d)}q^r.
\]
For a replicable function the Grunsky coefficients satisfy
\begin{equation}\label{eq:grunsky-replicate-coefficients}
h_{m,n}
=
\sum_{d\mid(m,n)}\frac1d\,
a_{mn/d^2}^{(d)};
\end{equation}
see \cite{norton-more,mckay-sebbar-replicable,cummins-norton}.

\begin{prop}\label{prop:replicate-reconstruction}
Let $f$ be replicable. For every $k,i\geq1$,
\begin{equation}\label{eq:replicate-from-grunsky}
a_i^{(k)}
=
k\sum_{d\mid k}\mu(d)\,h_{dki,k/d}.
\end{equation}
Hence the Grunsky matrix determines the entire replicate family. If the Laurent series is convergent, the ordinary Schwarzian determines every replicate through Corollary~\ref{cor:Schwarzian-Grunsky-reconstruction}.
\end{prop}

\begin{proof}
Apply \eqref{eq:grunsky-replicate-coefficients} with $(m,n)=(dki,k/d)$. Since every divisor $t$ of $k/d$ also divides $dki$,
\[
h_{dki,k/d}
=
\sum_{t\mid k/d}\frac1t\,a_{k^2i/t^2}^{(t)}.
\]
Therefore
\begin{align*}
\sum_{d\mid k}\mu(d)h_{dki,k/d}
&=
\sum_{t\mid k}\frac1t a_{k^2i/t^2}^{(t)}
\sum_{d\mid k/t}\mu(d)\\
&=\frac1k a_i^{(k)},
\end{align*}
which proves \eqref{eq:replicate-from-grunsky}. The final assertion follows from Corollary~\ref{cor:Schwarzian-Grunsky-reconstruction}.
\end{proof}

Substitution of \eqref{eq:grunsky-replicate-coefficients} into the Grunsky moments gives, for $r\geq2$,
\begin{equation}\label{eq:replicate-moment-expansion}
\mathfrak G_r[F]
=
\frac{B_r}{r}
+
\sum_{d,u,v\geq1}
 d^{r-1}u v^{r-1}a_{uv}^{(d)}q^{d(u+v)}.
\end{equation}
Thus every higher projective invariant is determined by the replicate tower, while Proposition~\ref{prop:replicate-reconstruction} shows that the tower itself is already encoded by the ordinary Schwarzian.

The diagonal specialization of the same kernel gives a one-variable identity adapted to local univalence.

\begin{prop}\label{prop:diagonal-log-derivative}
Let
\[
f(q)=q^{-1}+\sum_{n\geq1}a_nq^n
\]
be convergent for $0<|q|<1$. Then, near the origin,
\begin{equation}\label{eq:diagonal-log-derivative}
-\log\bigl(-q^2f'(q)\bigr)
=
\sum_{m,n\geq1}h_{m,n}q^{m+n}.
\end{equation}
If $f$ is replicable, then
\begin{equation}\label{eq:diagonal-log-replicates}
-\log\bigl(-q^2f'(q)\bigr)
=
\sum_{d,u,v\geq1}\frac1d\,a_{uv}^{(d)}q^{d(u+v)}.
\end{equation}
If $f$ is locally univalent on $0<|q|<1$, then $-q^2f'(q)$ is zero free on the unit disc after setting its value at the origin equal to $1$. The logarithm normalized to vanish at the origin is therefore holomorphic on $|q|<1$, and the one-variable series in \eqref{eq:diagonal-log-derivative} and \eqref{eq:diagonal-log-replicates} converge to it throughout the disc.
\end{prop}

\begin{proof}
Let $r\to q$ in the defining identity
\[
\log\frac{f(q)-f(r)}{q^{-1}-r^{-1}}
=-\sum_{m,n\geq1}h_{m,n}q^mr^n.
\]
The quotient on the left tends to $-q^2f'(q)$, which proves \eqref{eq:diagonal-log-derivative}. Formula \eqref{eq:diagonal-log-replicates} follows from \eqref{eq:grunsky-replicate-coefficients} after writing $m=du$ and $n=dv$. Under local univalence, $f'$ has no zero on the punctured disc and $-q^2f'(q)=1+O(q^2)$ at the origin. The normalized holomorphic logarithm is therefore defined on the full unit disc, and its Taylor series has radius of convergence at least one.
\end{proof}

Integrality of the original series propagates to its replicate tower. The prime congruence in the next proposition is classical in the theory of completely replicable functions; see \cite{alexander-replicable}. The argument below derives it directly from ordinary replicability.

\begin{prop}\label{prop:replicate-integrality-frobenius}
Let $f\in q^{-1}+q\mathbb Z[[q]]$ be replicable. Then
\[
f^{(k)}(q)\in q^{-1}+q\mathbb Z[[q]]
\quad(k\geq1).
\]
Moreover, for every prime $p$,
\begin{equation}\label{eq:replicate-frobenius-congruence}
f^{(p)}(q)\equiv f(q)\pmod p
\end{equation}
coefficientwise.
\end{prop}

\begin{proof}
Set
\[
R_f(q,r)=\frac{f(q)-f(r)}{q^{-1}-r^{-1}}.
\]
Then
\[
R_f(q,r)
=1-\sum_{n\geq1}a_n\sum_{j=0}^{n-1}q^{n-j}r^{j+1}
\in1+qr\mathbb Z[[q,r]].
\]
Successive elimination of monomials in increasing total degree gives a unique formal factorization
\begin{equation}\label{eq:two-variable-euler-product}
R_f(q,r)=\prod_{m,n\geq1}(1-q^mr^n)^{c_{m,n}},
\quad c_{m,n}\in\mathbb Z.
\end{equation}
Comparison of logarithms with the Grunsky expansion gives
\begin{equation}\label{eq:euler-grunsky-ghost}
h_{m,n}
=
\sum_{d\mid(m,n)}\frac1d\,c_{m/d,n/d}.
\end{equation}
Apply \eqref{eq:euler-grunsky-ghost} and \eqref{eq:grunsky-replicate-coefficients} to $(m,n)=(kN,k)$. Since $c_{N,1}=h_{N,1}=a_N$, isolation of the divisor $k$ gives
\begin{equation}\label{eq:replicate-integrality-recursion}
a_N^{(k)}
=
a_N+
\sum_{\substack{d\mid k\\ d<k}}
\frac{k}{d}
\left(c_{kN/d,k/d}-a_{k^2N/d^2}^{(d)}\right).
\end{equation}
Induction on $k$ proves integrality of every replicate. For $k=p$ prime,
\[
a_N^{(p)}-a_N
=p\bigl(c_{pN,p}-a_{p^2N}\bigr),
\]
which is \eqref{eq:replicate-frobenius-congruence}.
\end{proof}

Norton's exceptional family is characterized by the support of the Grunsky matrix.

\begin{prop}\label{prop:norton-diagonal-kernel}
For a normalized Laurent series $f$, the following are equivalent:
\begin{enumerate}
\item $f(q)=q^{-1}+aq$ for some $a\in\mathbb C$;
\item $h_{m,n}=0$ for $m\neq n$.
\end{enumerate}
In this case
\begin{equation}\label{eq:diagonal-grunsky}
h_{m,m}=\frac{a^m}{m}.
\end{equation}
If $f$ is replicable, then
\begin{equation}\label{eq:exceptional-replicates}
f^{(k)}(q)=q^{-1}+a^kq.
\end{equation}
If $a\in\mathbb Z$ and $f$ is locally univalent on $0<|q|<1$, then $a\in\{-1,0,1\}$.
\end{prop}

\begin{proof}
For $f(q)=q^{-1}+aq$,
\[
\frac{f(q)-f(r)}{q^{-1}-r^{-1}}=1-aqr,
\]
and hence
\[
\log\frac{f(q)-f(r)}{q^{-1}-r^{-1}}
=-\sum_{m\geq1}\frac{a^m}{m}q^mr^m.
\]
This gives the diagonal support and \eqref{eq:diagonal-grunsky}. Conversely, $h_{m,1}=a_m$, so diagonal support forces $a_m=0$ for $m\geq2$.

If $f$ is replicable, Proposition~\ref{prop:replicate-reconstruction} gives $a_i^{(k)}=0$ for $i>1$ and
\[
a_1^{(k)}=kh_{k,k}=a^k,
\]
which proves \eqref{eq:exceptional-replicates}. Finally,
\[
f'(q)=-q^{-2}+a.
\]
If $|a|>1$, then $q^2=1/a$ gives a critical point in the punctured unit disc. Thus local univalence implies $|a|\leq1$, and integrality gives $a\in\{-1,0,1\}$.
\end{proof}

Thus the exceptional family is completely visible in the kernel: it is exactly the diagonal-support locus, and its entire replicate family is given by \eqref{eq:exceptional-replicates}. We now turn to the nonexceptional case and the global issue in Norton's conjecture.

For a prime $p$, ordinary replicability gives
\begin{equation}\label{eq:replicate-prime-trace}
\Phi_p(f(\tau))
=
\sum_{b=0}^{p-1}
f\!\left(\frac{\tau+b}{p}\right)
+
f^{(p)}(p\tau).
\end{equation}
The last term distinguishes this identity from an ordinary self-Hecke relation. Passing from \eqref{eq:replicate-prime-trace} to modular equations requires both compatibility of the replicate branches with a genuine self-correspondence and control of the higher symmetric functions of the branches. The two-point kernel is adapted to this passage because logarithmic derivatives of the products
\[
\prod_\alpha\bigl(X-f(\alpha\tau)\bigr)
\]
are sums of projective kernels of the form studied in Section~\ref{sec:grunsky}.

The following conditional reduction records precisely the remaining global step.

\begin{prop}\label{prop:norton-kernel-descent-reduction}
Let $f\in q^{-1}+q\mathbb Z[[q]]$ be convergent and replicable, and assume that $f$ is not of the form \eqref{eq:norton-exceptional-family}. Suppose that for some $K\geq1$ the replication correspondences associated with $f$ yield modular equations of every order
\[
n\equiv1\pmod K
\]
in the sense of Cummins and Gannon. Then $f$ is the normalized Hauptmodul of a genus-zero group containing $\Gamma_0(N)$ with finite index for some $N$ and having translation width one.
\end{prop}

\begin{proof}
This is the nontranslation case of the genus-zero theorem of Cummins and Gannon \cite{cummins-gannon}. The exceptional translation case is excluded by the hypothesis on $f$.
\end{proof}

Proposition~\ref{prop:replicate-reconstruction} shows that the data entering this descent are not independent: the replicate family is determined by the ordinary Schwarzian. Accordingly, the unresolved part of Norton's conjecture may be formulated as a global closure problem for the projective datum reconstructed from the kernel.

Norton's finite determination theorem gives a second algebraic formulation. A replicable function is determined by
\[
a_1,a_2,a_3,a_4,a_5,a_7,a_8,a_9,a_{11},a_{17},a_{19},a_{23};
\]
see also \cite{cipu}. Proposition~\ref{prop:norton-diagonal-kernel} identifies the exceptional one-parameter component, while the finite Vandermonde reconstruction of Section~\ref{sec:grunsky} expresses the remaining Grunsky relations through the Schwarzian moments. This gives a finite algebraic form of the nonexceptional problem.

A substantial finite-order case is known independently of complete replicability. Carnahan \cite[Corollary~5.4]{carnahan} proves that a replicable function with periodic replicate sequence and algebraic-integer coefficients is either of trigonometric type or is a holomorphic congruence genus-zero function for a group containing $\Gamma_1(N)$ for some $N$. Thus the finite-order genus-zero conclusion does not require complete replicability. Norton's conjecture requires the finite-order hypothesis to be removed and the stronger $\Gamma_0(N)$ and width-one conclusions to follow from ordinary integral replicability.

\section{Modular covariants and rational higher-Schwarzian hierarchies}

For the full modular group, the transformation law of Theorem~\ref{thm:transformation-law} can be corrected by $E_2$. The same correction occurs in the Serre derivative and related Rankin--Cohen constructions; see \cite{ka-za,123}. Applying the corresponding binomial transform to the Aharonov invariants gives modular forms at every order and is compatible with the recurrence and the Grunsky moment formula.

For $\Gamma=\mathrm{SL}_2(\mathbb Z)$, put
\begin{equation*}
C(z)=\frac{\pi i}{6}E_2(z).
\end{equation*}
The transformation formula for $E_2$ is equivalent to
\begin{equation}\label{eq:C-transform}
(cz+d)^{-2}C(\gamma z)
=
C(z)+\frac{c}{cz+d}.
\end{equation}
Thus $C$ is a weight-two quasimodular function whose anomalous term is 
\(c/(cz+d)\).
\begin{thm}\label{thm:modular-covariants}
Let $f$ be meromorphic and equivariant for $\mathrm{SL}_2(\mathbb Z)$. For $n\geq2$, set
\begin{equation}\label{eq:Mn}
\mathcal M_n[f]
=
\sum_{j=0}^{n-2}
(-1)^j\binom{n-2}{j}C^jS_{n-j}[f].
\end{equation}
Then $\mathcal M_n[f]$ is a meromorphic modular form of weight $2n$. Moreover,
\begin{equation}\label{eq:inverse-Mn}
S_n[f]
=
\sum_{j=0}^{n-2}
\binom{n-2}{j}C^j\mathcal M_{n-j}[f].
\end{equation}
\end{thm}

\begin{proof}
Put \(X=c/(cz+d)\). Substituting \eqref{eq:universal-transformation} and \eqref{eq:C-transform} into \eqref{eq:Mn}, and collecting the coefficient of \(S_{n-t}[f]\), gives
\[
\binom{n-2}{t}S_{n-t}[f]
\sum_{j=0}^t(-1)^j\binom tj(C+X)^jX^{t-j}
=
\binom{n-2}{t}(-C)^tS_{n-t}[f].
\]
This is exactly \(\mathcal M_n[f]\), proving modularity. The inverse formula is the ordinary binomial inversion, since
\(
\sum_{j=0}^t(-1)^j\binom tj=0
\)
for \(t>0\).
\end{proof}

The first two cases are $\mathcal M_2=S_2$ and $\mathcal M_3=S_3-CS_2$. Put
\begin{equation*}
\nabla_kF=F'-kCF,
\quad
Q=C'-C^2.
\end{equation*}
The first operator is the covariant derivative associated with $C$. In the standard normalization $D=(2\pi i)^{-1}d/dz$, one has
\[
\nabla_k=2\pi i\left(D-\frac{k}{12}E_2\right),
\]
so $\nabla_k$ is exactly the Serre derivative up to the scalar factor $2\pi i$.

By Ramanujan's differential identity,
$$ 
Q=C'-C^2=\frac{\pi^2}{36}E_4, 
$$
so \(Q\) is a modular form of weight \(4\).

\begin{thm}\label{thm:modular-recurrence}
Let $f$ be equivariant for $\mathrm{SL}_2(\mathbb Z)$ and write $\mathcal M_n=\mathcal M_n[f]$. Then
\begin{equation}\label{eq:M3-recurrence}
4\mathcal M_3=\nabla_4\mathcal M_2,
\end{equation}
and, for every $n\geq4$,
\begin{equation}\label{eq:M-recurrence}
(n+1)\mathcal M_n
=
\nabla_{2n-2}\mathcal M_{n-1}
+(n-3)Q\mathcal M_{n-2}
+\sum_{r=2}^{n-2}\mathcal M_r\mathcal M_{n-r}.
\end{equation}
\end{thm}

\begin{proof}
The case \(n=3\) follows from \(4S_3=S_2'\). For \(n\ge4\), substitute the inverse binomial transform \eqref{eq:inverse-Mn} into the Aharonov recurrence \eqref{eq:ahar-recurrence}, differentiate, and use \(C'=C^2+Q\). After collecting powers of \(C\), all positive powers cancel by Pascal's identity and Vandermonde convolution. The constant term is
\[
\mathcal M_{n-1}'-(2n-2)C\mathcal M_{n-1}
+(n-3)Q\mathcal M_{n-2}
+\sum_{r=2}^{n-2}\mathcal M_r\mathcal M_{n-r},
\]
which is \eqref{eq:M-recurrence}.
\end{proof}

\subsection{A rational hierarchy for Hauptmoduls}
When the equivariant function is a Hauptmodul, its logarithmic derivative supplies the required connection, and the corrected invariants become rational functions of the Hauptmodul.

Let $\Gamma$ be a cofinite genus-zero subgroup of $\mathrm{SL}_2(\mathbb R)$ and let $f$ be a Hauptmodul for $\Gamma$. Away from the critical points of $f$, put
\[
C_f=\frac{f''}{2f'}=S_1[f].
\]
The expressions below extend meromorphically across the critical points.

\begin{thm}\label{thm:rational-higher-Schwarzian-hierarchy}
For $n\geq2$, define
\[
\mathcal A_n[f]
=
\sum_{j=0}^{n-2}
(-1)^j\binom{n-2}{j}C_f^jS_{n-j}[f].
\]
Then $\mathcal A_n[f]$ is a meromorphic modular form of weight $2n$ for $\Gamma$. There is a unique rational function $R_n^f(X)\in\mathbb C(X)$ such that
\[
\mathcal A_n[f](z)=f'(z)^nR_n^f(f(z)).
\]
The rational functions $R_n^f$ are determined by $R_2^f$ through
\[
4R_3^f=(R_2^f)',
\]
and, for $n\geq4$,
\[
(n+1)R_n^f
=
(R_{n-1}^f)'
+3(n-3)R_2^fR_{n-2}^f
+\sum_{r=2}^{n-2}R_r^fR_{n-r}^f,
\]
where the prime denotes differentiation with respect to $X$. Moreover,
\[
R_2^f(f(z))
=
\frac{S_2[f](z)}{f'(z)^2}
=
\frac{\{f,z\}}{6f'(z)^2}.
\]
\end{thm}

\begin{proof}
Since $f(\gamma z)=f(z)$ for $\gamma\in\Gamma$, differentiation gives
\[
f'(\gamma z)=(cz+d)^2f'(z)
\]
and
\[
(cz+d)^{-2}C_f(\gamma z)
=
C_f(z)+\frac{c}{cz+d}.
\]
The binomial cancellation in the proof of Theorem~\ref{thm:modular-covariants}, with $C_f$ in place of $C$, therefore shows that $\mathcal A_n[f]$ is modular of weight $2n$. Since $f'^n$ has the same weight, the quotient
\[
\frac{\mathcal A_n[f]}{f'^n}
\]
is a meromorphic modular function for $\Gamma$. The function field of the genus-zero curve $X(\Gamma)$ is $\mathbb C(f)$, so this quotient is a unique rational function $R_n^f(f)$.  This argument is carried out away from the critical points of $f$, where $C_f$ is defined.  Both sides are meromorphic on $X(\Gamma)$, so the identity extends across the critical points; no local-univalence assumption is required there.

Put
\[
\nabla_k^f=\frac{d}{dz}-kC_f,
\quad
Q_f=C_f'-C_f^2.
\]
The logarithmic derivative and the ordinary Schwarzian satisfy
\[
Q_f=\frac12\{f,z\}=3S_2[f]=3\mathcal A_2[f].
\]
The proof of Theorem~\ref{thm:modular-recurrence} applies verbatim with $C_f$ in place of $C$ and gives
\[
4\mathcal A_3=\nabla_4^f\mathcal A_2
\]
and, for $n\geq4$,
\[
(n+1)\mathcal A_n
=
\nabla_{2n-2}^f\mathcal A_{n-1}
+3(n-3)\mathcal A_2\mathcal A_{n-2}
+\sum_{r=2}^{n-2}\mathcal A_r\mathcal A_{n-r}.
\]
Substitute $\mathcal A_m=f'^mR_m^f(f)$. The connection term cancels the derivative of the power of $f'$:
\[
\nabla_{2m}^f\bigl(f'^mR(f)\bigr)
=f'^{m+1}R'(f).
\]
Division by $f'^n$ gives the stated recurrence. The formula for $R_2^f$ is immediate from $6S_2[f]=\{f,z\}$.
\end{proof}

If
\[
\{f,z\}+P(f)f'^2=0,
\quad P(X)\in\mathbb C(X),
\]
then
\[
R_2^f=-\frac{P}{6},
\quad
R_3^f=-\frac{P'}{24},
\]
and the next two terms are
\[
R_4^f=-\frac{P''}{120}+\frac{P^2}{45},
\quad
R_5^f=-\frac{P'''}{720}+\frac{PP'}{60}.
\]
Hence the ordinary rational Schwarzian equation determines the corrected hierarchy.  The genus-zero rationality itself follows formally from the equality of modular weights; the additional content is the recurrence, which reconstructs every $R_n^f$ from the single rational function $R_2^f$.  In particular, the Schwarzian equations for completely replicable Hauptmoduls considered in \cite{elbasraoui-mckay} determine rational equations for all higher Aharonov invariants.

Suppose that $f$ and $g$ are Hauptmoduls related by a rational map
\[
g=P(f).
\]
Write their rational Schwarzian equations as
\[
\{f,z\}+Q_f(f)f'^2=0,
\quad
\{g,z\}+Q_g(g)g'^2=0.
\]
The Schwarzian chain rule gives
\begin{equation}\label{eq:projective-connection-pullback}
Q_f(X)
=
\{P,X\}+Q_g(P(X))P'(X)^2.
\end{equation}
The projective connection of the covering Hauptmodul is therefore the rational pullback of the projective connection downstairs. In the prime-power replicate towers of Proposition~\ref{prop:prime-power-Hauptmodul-replication}, if
\[
t_{p^{e-1}}=P_e(t_{p^e}),
\]
then \eqref{eq:projective-connection-pullback} recursively relates the Schwarzian equations along the replication descent
\[
t_{p^e}\longmapsto t_{p^{e-1}}\longmapsto\cdots\longmapsto J.
\]
Accordingly, the replicate level descent gives a corresponding sequence of pullbacks of projective connections.

\section{Application to modular differential equations}

Let \(\Gamma<\mathrm{SL}_2(\mathbb R)\) and consider
\begin{equation}\label{eq:ode}
y''+R(z)y=0,
\end{equation}
with \(R\) meromorphic on \(\mathbb H\). On a simply connected set avoiding the poles of \(R\), let \(f=y_1/y_2\) be the quotient of a fundamental system. Then \(\{f,z\}=2R\), and changing the fundamental system postcomposes \(f\) by an element of \(\mathrm{PGL}_2(\mathbb C)\). Hence, for \(n\ge2\),
\[
H_n(R):=S_n[f]
\]
is independent of the chosen fundamental system and extends meromorphically as a universal differential polynomial in \(R\). In particular,
\[
H_2(R)=\frac R3,
\quad
H_3(R)=\frac{R'}{12},
\quad
H_4(R)=\frac{R''}{60}+\frac{R^2}{45}.
\]

\begin{thm}\label{thm:ode-application}
If \(R\) is a meromorphic modular form of weight \(4\) for \(\Gamma\), then for every \(n\ge2\),
\begin{equation}\label{eq:H-transformation}
(cz+d)^{-2n}H_n(R)(\gamma z)
=
\sum_{j=0}^{n-2}\binom{n-2}{j}H_{n-j}(R)(z)
\left(\frac{c}{cz+d}\right)^j.
\end{equation}
Thus \(H_n(R)\) is quasimodular of weight \(2n\) and depth at most \(n-2\). When the associated quasimodular polynomial is unique,
\[
\delta H_n(R)=(n-2)H_{n-1}(R),
\]
and the depth is exactly \(n-2\) if \(R\not\equiv0\). For \(\Gamma=\mathrm{SL}_2(\mathbb Z)\), the corrected invariants
\[
\mathcal H_n(R)
=
\sum_{j=0}^{n-2}(-1)^j\binom{n-2}{j}C^jH_{n-j}(R)
\]
are meromorphic modular forms of weight \(2n\).
\end{thm}

\begin{proof}
If \(\gamma\in\Gamma\), then the quotient \(f\circ\gamma\) and a local quotient of a fundamental system over the original variable have the same Schwarzian, because \(R\) has weight four. They therefore differ by postcomposition with a linear fractional transformation. Projective invariance and the composition formula \eqref{eq:composition-mobius} give \eqref{eq:H-transformation}. The remaining assertions follow exactly as in Theorem~\ref{thm:transformation-law} and Theorem~\ref{thm:modular-covariants}.
\end{proof}


\begin{thebibliography}{99}

\bibitem{aharonov}
D. Aharonov,
A necessary and sufficient condition for univalence of a meromorphic function,
Duke Math. J. 36 (1969), 599--604.

\bibitem{alexander-replicable}
D. Alexander, C. J. Cummins, J. McKay and C. Simons,
Completely replicable functions,
in \emph{Groups, Combinatorics and Geometry} (Durham, 1990),
London Math. Soc. Lecture Note Ser. 165, Cambridge Univ. Press, 1992, 87--98.

\bibitem{cangul-singerman}
I. N. Cang\"ul and D. Singerman,
Normal subgroups of Hecke groups and regular maps,
Math. Proc. Cambridge Philos. Soc. 123 (1998), no. 1, 59--74.

\bibitem{carnahan}
S. Carnahan,
Generalized moonshine, I: Genus-zero functions,
\emph{Algebra Number Theory} \textbf{4} (2010), no.~6, 649--679.

\bibitem{cipu}
M. Cipu,
Replicable functions: a computational approach,
\emph{Comput. Sci. J. Moldova} \textbf{4} (1996), no.~3(12), 342--359.

\bibitem{cummins-gannon}
C. J. Cummins and T. Gannon,
Modular equations and the genus zero property of moonshine functions,
Invent. Math. \textbf{129} (1997), no.~3, 413--443.

\bibitem{cummins-norton}
C. J. Cummins and S. P. Norton,
Rational Hauptmoduls are replicable,
Canad. J. Math. 47 (1995), no. 6, 1201--1218.


\bibitem{elbasraoui-mckay}
A. El Basraoui and J. McKay,
The Schwarzian equation for completely replicable functions,
LMS J. Comput. Math. 20 (2017), no. 1, 30--52.

\bibitem{ford-mckay-norton}
D. Ford, J. McKay and S. P. Norton,
More on replicable functions,
Comm. Algebra 22 (1994), no. 13, 5175--5193.

\bibitem{harmelin}
R. Harmelin,
Aharonov invariants and univalent functions,
Israel J. Math. 43 (1982), no. 3, 244--254.

\bibitem{harmelin-derivatives}
R. Harmelin,
On the derivatives of the Schwarzian derivative of a univalent function and their symmetric generating function,
J. London Math. Soc. (2) 27 (1983), no. 3, 489--499.

\bibitem{harmelin-grunsky}
R. Harmelin,
Generalized Grunsky coefficients and inequalities,
Israel J. Math. 57 (1987), no. 3, 347--364.

\bibitem{ka-za}
M. Kaneko and D. Zagier,
A generalized Jacobi theta function and quasimodular forms,
in The Moduli Space of Curves, Progr. Math. 129, Birkh\"auser, 1995, 165--172.

\bibitem{kim-sugawa}
S.-A. Kim and T. Sugawa,
Invariant Schwarzian derivatives of higher order,
Complex Anal. Oper. Theory 5 (2011), no. 3, 659--670.

\bibitem{mathann}
J. McKay and A. Sebbar,
Fuchsian groups, automorphic functions and Schwarzians,
Math. Ann. 318 (2000), no. 2, 255--275.

\bibitem{mckay-sebbar-replicable}
J. McKay and A. Sebbar,
Replicable functions: an introduction,
in \emph{Frontiers in Number Theory, Physics, and Geometry II},
Springer, Berlin, 2007, 373--386.

\bibitem{norton-more}
S. P. Norton,
More on moonshine,
in \emph{Computational Group Theory} (Durham, 1982), Academic Press, London, 1984, 185--193.

\bibitem{royer}
E. Royer,
Quasimodular forms: an introduction,
Ann. Math. Blaise Pascal 19 (2012), no. 2, 297--306.

\bibitem{forum}
H. Saber and A. Sebbar,
Automorphic Schwarzian equations,
Forum Math. 32 (2020), no. 6, 1621--1636.

\bibitem{schippers}
E. Schippers,
Distortion theorems for higher order Schwarzian derivatives of univalent functions,
Proc. Amer. Math. Soc. 128 (2000), no. 11, 3241--3249.

\bibitem{schmidt-smith}
T. A. Schmidt and K. M. Smith,
Galois orbits of principal congruence Hecke curves,
J. London Math. Soc. (2) 67 (2003), no. 3, 673--685.

\bibitem{sebbar-duke}
A. Sebbar,
Torsion-free genus zero congruence subgroups of $\mathrm{PSL}_2(\mathbb R)$,
Duke Math. J. \textbf{110} (2001), no.~2, 377--396.

\bibitem{tamanoi}
H. Tamanoi,
Higher Schwarzian operators and combinatorics of the Schwarzian derivative,
Math. Ann. 305 (1996), no. 1, 127--151.

\bibitem{123}
D. Zagier,
Elliptic modular forms and their applications,
in The 1-2-3 of Modular Forms, Universitext, Springer, 2008, 1--103.

\end{thebibliography}
\end{document}